\documentclass[12pt]{amsart}
\usepackage[utf8]{inputenc}
\usepackage{amsmath}
\usepackage{amsthm}
\usepackage{amssymb,colonequals,mathrsfs,mathtools}
\usepackage{array,enumitem,yfonts}
\usepackage{comment}
\usepackage[alphabetic]{amsrefs}
\usepackage[all,cmtip]{xy}
\usepackage[colorlinks,anchorcolor=blue,citecolor=blue,linkcolor=blue,urlcolor=blue,bookmarksopen=true]{hyperref}
\allowdisplaybreaks

\usepackage[margin=1in]{geometry}

\usepackage{tikz}
\usetikzlibrary{cd,arrows}
\tikzset{>=stealth}
\tikzcdset{arrow style=tikz}
\tikzset{link/.style={column sep=1.8cm,row sep=0.16cm}}

\AtBeginDocument{%
	\def\MR#1{}
}
\usepackage{fancyhdr}
\fancypagestyle{titlepage}
{
	\fancyhf{}

	\fancyfoot[l]{
	\href{https://mathscinet.ams.org/mathscinet/msc/msc2020.html}
	}
}

\newcommand{\bZ}{\mathbb{Z}}    
\newcommand{\bF}{\mathbb{F}}    
\newcommand{\Gr}{\operatorname{Gr}} 
\newcommand{\bA}{\mathbb{A}}    
\newcommand{\bP}{\mathbb{P}}    
\newcommand{\Mon}{\operatorname{Mon}}    
\newcommand{\Gal}{\operatorname{Gal}} 
\newcommand{\GL}{\operatorname{GL}}

\newcommand{\Hom}{\operatorname{Hom}}
\newcommand{\rk}{\operatorname{rank}}
\newcommand{\Frob}{\operatorname{Frob}}

\newtheorem{thm}{Theorem}[section]
\newtheorem{prop}[thm]{Proposition}
\newtheorem{lemma}[thm]{Lemma}
\newtheorem{cor}[thm]{Corollary}
\numberwithin{equation}{section}

\theoremstyle{definition}

\newtheorem{defn}[thm]{Definition}

\newtheorem{rmk}[thm]{Remark}

\begin{document}

\title{A geometric approach to the density of rank-metric codes}
\author{Shamil~Asgarli
    \and Lian~Duan
    \and Nathan~Kaplan
    \and Kuan-Wen~Lai}

\newcommand{\ContactInfo}{{
\bigskip\footnotesize

\medskip
\noindent S.~Asgarli,
\textsc{Department of Mathematics and Computer Science \\
Santa Clara University \\ 
Santa Clara, CA 95050, USA}\par\nopagebreak
\noindent\texttt{sasgarli@scu.edu}

\bigskip
\noindent L.~Duan,
\textsc{Institute of Mathematical Sciences\\
ShanghaiTech University\\
No.393 Middle Huaxia Road, Pudong New District, 
Shanghai, China}\par\nopagebreak
\noindent\textsc{Email:} \texttt{duanlian@shanghaitech.edu.cn}

\bigskip
\noindent N.~Kaplan,
\textsc{Department of Mathematics \\
University of California, Irvine \\ 
Irvine, CA 92697, USA}\par\nopagebreak
\noindent\texttt{nckaplan@math.uci.edu}

\bigskip
\noindent K.-W.~Lai,
\textsc{Department of Smart Computing and Applied Mathematics \\
Tunghai University \\
No.~1727, Sec.~4, Taiwan~Blvd., Xitun~Dist., Taichung~City 407224, Taiwan}
\par\nopagebreak
\smallskip
\noindent\textsc{National Center for Theoretical Science \\
No.~1, Sec.~4, Roosevelt~Rd., Taipei~City 106319, Taiwan}
\par\nopagebreak
\noindent\textsc{Email:} \texttt{kwlai@thu.edu.tw}
}}

\begin{abstract}
We study the asymptotic density of $\mathbb{F}_q$-point-free linear sections of geometrically irreducible projective varieties over finite fields. We then apply these results to rank-metric codes via determinantal varieties. Our approach recovers the known cases in which the density tends to $0$ or $1$ and determines the limit in the cases where it was previously unknown. To compute these previously unknown limits, we extend the notion of quasireflexivity to higher-dimensional varieties and show that determinantal varieties satisfy this property. This allows us to invoke the Chebotarev density theorem for varieties over finite fields to obtain the desired estimate.
\end{abstract}

\maketitle
\thispagestyle{titlepage}

\setcounter{tocdepth}{2}
\makeatletter
\def\l@subsection{\@tocline{2}{0pt}{2.3pc}{5pc}{}}
\makeatother


\section{Introduction}
\label{sec:intro}

Let $\bF_q$ be the finite field with $q$ elements, where $q$ is a prime power, and regard $\bF_q^{mn}$ as the space of $m\times n$ matrices over $\bF_q$ equipped with the rank metric
$$
    d(A,B) \colonequals \rk(A-B), \qquad
    A, B\in\bF_q^{mn}.
$$
Throughout the paper, we assume without loss of generality that $m\geq n$.  A \emph{linear rank-metric code} is a linear subspace $V\subseteq\bF_q^{mn}$. One of the most important invariants of a nonzero code is its \emph{minimum (rank) distance}, defined by
\begin{align*}
    d_R(V) \colonequals & \min\left\{
        d(A, B) \;\middle|\; A, B \in V,\; A\neq B
    \right\} \\
    =& \min\left\{
        \rk(M) \;\middle|\; M \in V\setminus\{0\}
    \right\}.
\end{align*}
The minimum distance determines the error-correcting capabilities of the code.

Rank-metric codes have become a central object of study in modern coding theory. They arise naturally in contexts where errors are measured by rank, with important applications to random network coding and secure network coding \cites{SilvaKS08,SilvaK11}. They are also closely related to several areas of pure mathematics, including $q$-polymatroids and finite geometry \cites{GJLR20,CMPZ17}. For a general introduction to the subject, we refer the reader to the surveys \cites{GKR23,BHLPRW22}.

The main problem that we study is motivated by analogous results for linear codes in the Hamming metric. There has been extensive study of the distribution of the minimum distance among all $k$-dimensional linear codes $C \subseteq \bF_q^n$. For a recent notable example, see \cite{HHLT22}.  Much of the work in this area has focused on the asymptotic problem where $q$ is fixed, while $k$ and $n$ grow together in such a way that $k/n$ converges. Loidreau studied the analogue of this problem for rank-metric codes in \cite{Loi14}.

We are interested in a different kind of limit where $q$ tends to infinity while the other parameters remain fixed. Recall that if $C \subseteq \bF_q^n$ is a $k$-dimensional linear code with minimum Hamming distance $d$, then the Singleton bound gives $d \le n-k+1$. Codes for which equality holds are called \emph{MDS} (maximum distance separable). When $q\to\infty$ with $k$ and $n$ fixed, it is not difficult to show that most $k$-dimensional linear codes $C\subseteq \bF_q^n$ are MDS. For a quantitative result in this direction, see \cite{Kai14}.

We phrase our main results in terms of the following density. Fix integers $1\leq k\leq mn$ and $\delta\geq 1$, and consider the density of $k$-dimensional codes with minimum rank distance at least $\delta$:
$$
    D_q(m\times n, k, \delta) \colonequals \frac{\left|\{
        \text{codes }V\subseteq\bF_q^{mn}
    \mid
        \dim V = k,\;
        d_R(V)\geq\delta
    \}\right|}{\left|\{
        \text{codes }V\subseteq\bF_q^{mn}
    \mid
        \dim V = k
    \}\right|}.
$$
The rank-metric analogue of the Singleton bound asserts that for a $k$-dimensional linear rank-metric code $V \subseteq\bF_q^{mn}$, we have $k \le m(n-d_R(V) + 1)$ \cite{GKR23}*{Theorem~3.1}. Codes that reach equality in this bound are called \emph{MRD} (maximum rank distance). They are the rank-metric analogues of MDS codes.

There has been significant recent interest in density questions for MRD codes. For an overview of this topic, see \cite{GKR23}*{Section~6}. Gruica and Ravagnani have determined the asymptotic density of MRD codes for all sets of parameters.

\begin{thm}[\cite{GR22}*{Theorem~5.9}]
Assume that $1\leq\delta\leq n$. Then
$$
    \lim_{q\to\infty} D_q(
        m \times n,\,
        m(n-\delta+1), \, \delta
    ) = \begin{cases}
        1 & \text{ if } \delta = 1,\\
        \sum_{i=0}^m \frac{(-1)^i}{i!} & \text{ if } n = \delta = 2,\\
        0 & \text{ otherwise.}
    \end{cases}
$$
\end{thm}

\noindent The $n=\delta=2$ case is due to Antrobus and Gluesing-Luerssen \cite{AGL19}*{Corollary~VII.5}.

\medskip

Gruica and Ravagnani apply their techniques to obtain upper and lower bounds on the density of rank-metric codes with more general minimum distance and dimension. This allows them to determine the behavior of $D_q(m\times n, k, \delta)$ as $q\to \infty$ in most cases.

\begin{thm}[\cite{GR22}*{Theorem~5.12}]
Assume that $1\leq k\leq mn$ and $1\leq\delta\leq n$. Then
$$
    \lim_{q\to\infty}D_q(m\times n, k,\delta) = \begin{cases}
        0 & \text{if}\quad(\delta-1)(m+n-\delta+1) > mn-k+1, \\
        1 & \text{if}\quad(\delta-1)(m+n-\delta+1) < mn-k+1.
    \end{cases}
$$
\end{thm}

In view of this result, we refer to the first situation as the \emph{sparse case} and the second as the \emph{abundant case}.  The remaining situation is the \emph{threshold case}, namely when
$$
    (\delta-1)(m+n-\delta+1) = mn-k+1.
$$
Gruica and Ravagnani explain that in the threshold case, \cite{GR22}*{Proposition~4.4} implies the upper bound
$$
\limsup_{q\to\infty} D_q(m\times n, k, \delta) \le \frac{1}{2}.
$$
They further remark that their techniques do not determine the limit in this case and that, in general, rank-metric codes in this setting are neither sparse nor dense as $q$ grows  \cite{GR22}*{Remark~5.13}. 

\subsection{Main results}

In this paper, we adopt an algebro-geometric approach to these density questions. Our main theorem (Theorem~\ref{thm:main}) resolves the threshold case as a corollary, and also yields concise proofs for the sparse and abundant cases.

\begin{cor}
\label{cor:main}
Fix integers $m\geq n$, $2\leq\delta\leq n$, and $1\leq k\leq mn$, and set
$$
\Delta\colonequals \prod_{i=0}^{n-\delta} \frac{\binom{m+i}{\delta-1}}{\binom{\delta-1+i}{\delta-1}}.
$$
Then the asymptotic density of $k$-dimensional codes in $\bF_q^{mn}$ with minimum rank distance at least $\delta$ is given by
$$
    \lim_{q\to\infty}D_q(m\times n, k, \delta) = \begin{cases}
        0 & \text{if}\quad(\delta-1)(m+n-\delta+1) > mn-k+1, \\
        1 & \text{if}\quad(\delta-1)(m+n-\delta+1) < mn-k+1, \\
        \sum_{i=0}^{\Delta}\frac{(-1)^i}{i!}
            & \text{if}\quad(\delta-1)(m+n-\delta+1) = mn-k+1.
    \end{cases}
$$
For completeness, when $1\leq k\leq mn$, we have $D_q(m\times n,k,1)=1$, while $D_q(m\times n,k,\delta)=0$ for every $\delta>n$.
\end{cor}

Our main theorem concerns the density of linear subspaces whose
intersections with a geometrically irreducible projective variety contain no $\bF_q$-rational points. We apply this result to determinantal varieties in order to study rank-metric codes. We now introduce the background necessary to state this result precisely. Let $\delta$ be an integer satisfying $2 \le \delta \le n$, and set $W\colonequals\operatorname{Mat}_{m\times n}(\bF_q)$. Inside $\bP(W)\cong\bP_{\bF_q}^{mn-1}$, let $Y_\delta$ be the projective determinantal variety cut out by the $\delta\times\delta$ minors. More concretely, we can describe the points of $Y_{\delta}$ as follows. For every field extension $K/\bF_q$, its $K$-points are
\[
    Y_\delta(K)
    =
    \left\{
        [M]\in\bP^{mn-1}(K)
        \;\middle|\;
        \rk_K(M)<\delta
    \right\}.
\]
By definition, a code $V\subseteq\bF_q^{mn}$ has $d_R(V)\geq\delta$ if and only if the linear section $\bP(V)\cap Y_\delta$ contains no $\bF_q$-point. Moreover, \cite{Har92}*{Proposition~12.2} gives \begin{align*}
\dim Y_\delta &=(mn-1)-(m-\delta+1)(n-\delta+1)\\ &=(\delta-1)(m+n-\delta+1)-1.
\end{align*}
The dimension of a variety can be characterized by its general linear sections \cite{Har92}*{Definitions~11.2--11.3}. Accordingly, for a general linear subspace $V\subseteq\overline{\bF_q}^{mn}$ of dimension~$k$,
$$
    \bP(V)\cap (Y_\delta)_{\overline{\bF_q}}\;\text{ is }\begin{cases}
        \text{positive-dimensional}
        &\text{if}\quad (\delta-1)(m+n-\delta+1) > mn-k+1, \\
        \text{empty}
        &\text{if}\quad (\delta-1)(m+n-\delta+1) < mn-k+1, \\
        \text{zero-dimensional}
        &\text{if}\quad (\delta-1)(m+n-\delta+1) = mn-k+1,
    \end{cases}
$$
which correspond precisely to the sparse, abundant, and threshold cases, respectively.

\smallskip

Determinantal varieties are examples of \emph{geometrically irreducible varieties}, meaning that they remain irreducible in the Zariski topology after extending the base field to the algebraic closure $\overline{\bF_q}$; see, for example, \cite{Har92}*{Proposition~12.2}. They also satisfy \emph{quasireflexivity}, a notion introduced for curves by Entin \cite{Ent21}*{Section~2}. More precisely, a possibly reducible curve $C\subseteq\bP^n$ of degree~$d$ is quasireflexive if, for every irreducible component $C_i$ of $C$, there exists a hyperplane tangent to $C_i$ whose set-theoretic intersection with $C$ consists of $d-1$ geometric points. Thus, the tangency is as simple as possible, while all other intersections are transverse. We extend this notion to higher dimensions, restricting attention
to geometrically irreducible varieties.

\begin{defn}
\label{def:quasireflex}
Let $K$ be a field, and let
$Y\subseteq\bP^N_K$ be a geometrically irreducible projective
variety of degree $\Delta\geq2$. We say that $Y$ is
\emph{quasireflexive} if there exists a linear subspace
$L\subseteq\bP^N_{\overline K}$ of dimension $N-\dim Y$ such that
the scheme-theoretic intersection $Z\colonequals L\cap Y_{\overline K}$ is zero-dimensional and its support consists of exactly $\Delta-1$ distinct geometric points, all lying in the smooth locus of $Y_{\overline K}$.
\end{defn}

Since $Z$ is a proper complementary-dimensional linear section,
B\'ezout's theorem gives
\[
    \sum_{y\in\operatorname{Supp}(Z)}
    \dim_{\overline K}\mathcal{O}_{Z,y}
    =\Delta.
\]
Consequently, after relabeling the points as $\operatorname{Supp}(Z) =\{y_0,y_1,\ldots,y_{\Delta-2}\}$, 
we have
\[
    \dim_{\overline K}\mathcal{O}_{Z,y_0}=2
    \qquad\text{and}\qquad
    \dim_{\overline K}\mathcal{O}_{Z,y_i}=1
    \quad\text{for }1\leq i\leq\Delta-2.
\]
Thus, exactly one intersection point has multiplicity two, and all
the remaining intersections are transverse.

Quasireflexivity enables us to apply the \emph{Chebotarev density
theorem}; see Section~\ref{sec:threshold} for further details. To make the asymptotic statements below precise, we work with
varieties obtained by base change from a fixed scheme over the
integers. Let
\[
    \mathscr{Y}\subseteq\bP^N_{\bZ}
\]
be a fixed closed projective subscheme defined by finitely many
homogeneous polynomials with integer coefficients. For each prime
power $q=p^r$, let
\[
    Y_q
    \colonequals
    \mathscr{Y}
    \times_{\operatorname{Spec}(\bZ)}
    \operatorname{Spec}(\bF_q)
    \subseteq\bP^N_{\bF_q},
\]
where the fiber product is taken with respect to the canonical map
$\operatorname{Spec}(\bF_q)\to\operatorname{Spec}(\bZ)$. Thus, the
varieties $Y_q$ are the finite-field base changes of the single fixed
scheme $\mathscr{Y}$.

We assume that, for every prime power $q$, the scheme $Y_q$ is a
geometrically irreducible projective variety, and that its dimension
$d$ and degree $\Delta$ are independent of $q$. All limits
$q\to\infty$ below are taken over prime powers. When a single prime
power $q$ is fixed, we may abbreviate $Y_q$ to $Y$.

The main result of this paper is as follows.

\begin{thm}
\label{thm:main}
Let $\mathscr{Y}\subseteq\bP^N_{\bZ}$ and its base changes $Y_q$ be
as above, so that every $Y_q$ is a geometrically irreducible
projective variety of dimension $d$ and degree $\Delta$. Let
$0\leq\ell\leq N$. For every prime power $q$, define
\[
    D_q(Y_q,\ell)
    \colonequals
    \frac{
        \left|\left\{
            L\in\Gr(\bP^\ell,\bP^N)(\bF_q)
            \;\middle|\;
            (L\cap Y_q)(\bF_q)=\emptyset
        \right\}\right|
    }{
        \left|
            \Gr(\bP^\ell,\bP^N)(\bF_q)
        \right|
    }.
\]
\begin{enumerate}[label=\textup{(\arabic*)}]
\item There exists a constant $c_1$, depending only on
$\mathscr{Y}$ and $\ell$, such that, for every prime power $q$,
\[
    0\leq D_q(Y_q,\ell)
    \leq c_1q^{N-d-\ell}.
\]
In particular, if $d+\ell>N$, then
\[
    \lim_{q\to\infty}D_q(Y_q,\ell)=0.
\]

\item There exists a constant $c_2$, depending only on
$\mathscr{Y}$ and $\ell$, such that, for every prime power $q$,
\[
    1-c_2q^{d+\ell-N}
    \leq D_q(Y_q,\ell)\leq1.
\]
In particular, if $d+\ell<N$, then
\[
    \lim_{q\to\infty}D_q(Y_q,\ell)=1.
\]

\item Suppose that $\Delta\geq2$, $d+\ell=N$, and that every
$Y_q$ is quasireflexive. Then
\[
    \lim_{q\to\infty}D_q(Y_q,\ell)
    =
    \sum_{i=0}^{\Delta}\frac{(-1)^i}{i!}.
\]
\end{enumerate}
\end{thm}

For the application to rank-metric codes, fix $m$, $n$, and $\delta$,
and let
\[
    \mathscr{Y}_{\delta}
    \subseteq\bP_{\bZ}^{mn-1}
\]
be the projective determinantal scheme cut out by the
$\delta\times\delta$ minors of the generic $m\times n$ matrix. These minors have integer coefficients, so this defines a projective
scheme over $\bZ$. For every prime power $q$, its base change
\[
    Y_{\delta,q}
    \colonequals
    \mathscr{Y}_{\delta}
    \times_{\operatorname{Spec}(\bZ)}
    \operatorname{Spec}(\bF_q)
    \subseteq\bP_{\bF_q}^{mn-1}
\]
is the determinantal variety of $m\times n$ matrices over $\bF_q$ of
rank less than $\delta$.

The dimension and degree of these fibers are independent of $q$, and
we prove in Theorem~\ref{thm:det-var_quasireflex} that they are
quasireflexive in every characteristic. We therefore apply
Theorem~\ref{thm:main} to $\mathscr{Y}_{\delta}$ with
\[
    N=mn-1,
    \qquad
    \ell=k-1.
\]
The degree formula for determinantal varieties
(see \cite{Ful98}*{Example~14.4.11}) gives
\[
    \Delta
    =\deg(Y_{\delta,q})
    =
    \prod_{i=0}^{n-\delta}
    \frac{
        \binom{m+i}{\delta-1}
    }{
        \binom{\delta-1+i}{\delta-1}
    }.
\]
Moreover,
\[
    D_q(m\times n,k,\delta)
    =
    D_q(Y_{\delta,q},k-1).
\]
Thus, Theorem~\ref{thm:main} directly yields
Corollary~\ref{cor:main}.

\subsection{Related work}

We now discuss how our method relates to previous approaches to density questions for rank-metric codes. Byrne and Ravagnani developed a combinatorial approach based on families of codes that are balanced with respect to a given partition of the ambient space \cite{BR20}. For MRD codes with $n = \delta = 2$, Antrobus and Gluesing-Luerssen utilized enumerative results on square matrices over $\bF_q$ with no eigenvalue in $\bF_q$ \cite{AGL19}. Gluesing-Luerssen deduced an exact formula for $D_q(3\times 3,3,3)$ via the connection between square full-rank MRD codes and the theory of semifields \cite{GL20}*{Theorem~2.4}. This connection was further developed by Gruica, Ravagnani, Sheekey, and Zullo in \cite{GRSZ23}*{Section~3}.

Gruica and Ravagnani \cite{GR22} study density questions by counting
common complements of collections of linear subspaces. As explained in
\cite{GKR23}*{Remark~6.6~(i)}, every matrix
$M\in\bF_q^{mn}$ with $\operatorname{rank}(M)<\delta$ has row space
contained in some $(\delta-1)$-dimensional subspace
$U\subseteq\bF_q^n$. For each such subspace, define
\[
    \bF_q^{mn}(U)
    \colonequals
    \left\{
        M\in\operatorname{Mat}_{m\times n}(\bF_q)
        \;\middle|\;
        \operatorname{rowsp}(M)\subseteq U
    \right\}.
\]
This is a linear subspace of $\bF_q^{mn}$ of dimension
$m(\delta-1)$. A code $C\subseteq\bF_q^{mn}$ of dimension
$m(n-\delta+1)$ has minimum rank distance at least $\delta$ if and
only if $C\cap\bF_q^{mn}(U)=\{0\}$ 
for every $(\delta-1)$-dimensional subspace $U\subseteq\bF_q^n$.
Since
\[
    \dim C+\dim\bF_q^{mn}(U)
    =m(n-\delta+1)+m(\delta-1)
    =mn,
\]
this condition is equivalent to
\[
    \bF_q^{mn}
    =C\oplus\bF_q^{mn}(U)
\]
for every such $U$. Thus, the MRD codes of dimension
$m(n-\delta+1)$ and minimum rank distance $\delta$ correspond to common complements of the spaces $\bF_q^{mn}(U)$.
The problem of estimating the number of such common complements can
then be turned into one of counting isolated vertices in a certain
bipartite graph. This approach fits into the study of the Critical
Problem for combinatorial geometries. For a more extensive discussion
of this connection, see \cite{GRSZ23}*{Section~5}.

In some ways, our geometric approach using linear sections of determinantal varieties is similar to the common-complement approach described above. The key distinction is that, rather than viewing the set of $\bF_q$-matrices of rank at most $\delta-1$ as a union of finitely many linear subspaces, we work with the associated irreducible projective variety, exploiting more of the
geometry inherent in the problem.

The previous work most closely related to our approach concerns a
different kind of limit. In \cite{NHTRR18}, Neri,
Horlemann-Trautmann, Randrianarisoa, and Rosenthal study the
proportion of $\bF_{q^s}$-linear MRD codes as the field extension
degree $s$ tends to infinity. Let
$C\subseteq\bF_{q^s}^n$ be an $\bF_{q^s}$-linear MRD code of
dimension $k$. After elementary row operations on a generator matrix
and a permutation of the $n$ coordinates, such a code admits a
generator matrix in \emph{systematic form}
\[
    G=(I_k\mid X).
\]
Building on a general criterion for generator matrices of MRD codes
\cite{NHTRR18}*{Proposition~13}, the authors show that the
$k\times(n-k)$ matrices $X$ yielding an MRD code form a Zariski open
subset of the affine space of all such matrices
\cite{NHTRR18}*{Theorem~17}. Whereas their parameter space is an
affine space of matrices, our parameter space is the Grassmannian of
$k$-dimensional subspaces of the ambient matrix space, allowing us to
exploit its intrinsic projective geometry.

\medskip

We end the introduction with several directions not pursued in this paper:

\vspace{-3pt}
\subsubsection*{Rate of convergence}

Corollary~\ref{cor:main} computes the limit of $D_q(m\times n, k, \delta)$ as $q \to \infty$, but it does not address the rate of convergence. As an illustrative example, in the case where $n=m=k=\delta=3$, our arguments yield an upper bound of the form $O(q^{-1})$ for the density, whereas Gluesing-Luerssen, using connections to the theory of semifields, obtained an exact formula that is $O(q^{-3})$ \cite{GL20}*{Theorem~2.4}. Although Theorem~\ref{thm:main} provides some information about the rate of convergence, we have not attempted to optimize it here.

\vspace{-3pt}
\subsubsection*{Field of linearity}

After fixing a basis of $\bF_{q^m}$ over $\bF_q$, each element of $\bF_{q^m}$ can be identified with a vector in $\bF_q^m$, and hence each element of $\bF_{q^m}^n$ with an $m \times n$ matrix over $\bF_q$. Thus, if $\ell$ divides $m$, an $\bF_{q^\ell}$-linear subspace of $\bF_{q^m}^n$ corresponds to an $\bF_q$-linear subspace of the space of $m \times n$ matrices. Note that every $\bF_{q^\ell}$-linear code is $\bF_q$-linear, whereas the converse does not hold in general. This issue is discussed explicitly in \cite{GHRW24}*{Remark~5.2}.

There has been considerable interest in density questions for such $\bF_{q^\ell}$-linear rank-metric codes. In \cite{GHRW24}, Gruica, Horlemann, Ravagnani, and Willenborg investigate these questions in a broader setting. In particular, they consider an analogue of $D_q(m\times n, k, \delta)$ with an additional parameter specifying that only $\bF_{q^\ell}$-linear subspaces are counted. They also study analogous density questions for codes endowed with other metrics, including the sum-rank metric.

\vspace{-3pt}
\subsubsection*{Other limits of interest}

In \cite{NHTRR18}, Neri, Horlemann-Trautmann, Randrianarisoa, and Rosenthal focus on the $\bF_{q^m}$-linear rank-metric codes described above and show that $\bF_{q^m}$-linear MRD codes are dense as $m \to \infty$. This result was generalized by Gruica, Horlemann, Ravagnani, and Willenborg, who proved that $\bF_{q^\ell}$-linear MRD codes in $\bF_{q^{\ell s}}^n$ are dense as $\ell \to \infty$ \cite{GHRW24}*{Theorem~5.8}.

Our geometric arguments, by contrast, require the full set of $k$-dimensional $\bF_q$-linear subspaces of the space of $m \times n$ matrices over $\bF_q$, rather than only those that are linear over some extension field $\bF_{q^\ell}$. We also require the dimension of the matrix space to remain fixed as $q \to \infty$. Consequently, our results do not address other limits of interest, such as the limit of $D_q(m\times n, k, \delta)$ as $m \to \infty$.

\subsection*{Organization of the paper}
The first two cases of Theorem~\ref{thm:main} are proved in Section~\ref{sec:sparse-abd}, while the remaining threshold case is treated in Section~\ref{sec:threshold}. In Section~\ref{sec:quasiref}, we verify that determinantal varieties are quasireflexive. In Appendix~\ref{sec:reflexive-det}, we discuss the closely related notion of reflexivity, and prove that determinantal varieties satisfy this property in arbitrary characteristic.

\subsection*{Acknowledgments}
The third author is supported by NSF grant DMS 2154223.  He thanks Tianhao Wang for many helpful conversations related to this project. The fourth author is supported by the NSTC Research Grant 113-2115-M-029-003-MY3 and partially supported by the Academic Summit Program 114-2639-M-002-009-ASP. 

\subsection*{AI disclosure}
The main results and the overall proof strategy were developed by the authors, who completed a draft of the paper before using any AI tools. ChatGPT (models 5.5 Thinking and 5.6 Sol) was subsequently used during the revision process. It assisted in reviewing mathematical arguments, identifying errors in the draft, suggesting corrections and proof strategies, and improving the clarity and presentation of the manuscript. For example, the proofs of Lemmas~\ref{lem:double-not-triple} and \ref{lem:multiplicity-2-open} are based on suggestions from ChatGPT. The authors carefully reviewed all AI-generated suggestions, independently verified the mathematical arguments and references, and made suitable adjustments to the final wording. The authors take full responsibility for the content of the paper.

\section{Asymptotic density in the sparse and abundant cases}
\label{sec:sparse-abd}

We prove the sparse and abundant cases of Theorem~\ref{thm:main} in this section. For each prime power $q$, we write $Y=Y_q$. All implied constants below are uniform in $q$. Indeed, the fibers $Y_q$, the associated incidence correspondences, and their projection images have complexity bounded independently of $q$, as they arise from the fixed equations defining $\mathscr Y\subseteq\bP^N_{\bZ}$. We use point-counting estimates in their bounded-complexity form; see \cite{Ent21}*{Section~4} and Remark~\ref{rmk:complexity} for more details.

\subsection{The sparse case}

For a projective variety $Y\subseteq\bP^N_{\bF_q}$, choose
$L\in\Gr(\bP^\ell,\bP^N)(\bF_q)$ uniformly at random and define
$$
\xymatrix@R=0pt{ \Theta\colon\Gr(\bP^\ell,\bP^N)(\bF_q)\ar[r]
& \bZ_{\geq 0} \\
L\ar@{|->}[r] & |(L\cap Y)(\bF_q)|.
} 
$$
We view $\Theta$ as a random variable. As shown for hyperplanes in \cite{PS22}*{Lemma~4.1} and for higher-codimensional linear sections in \cite{KS22}*{Lemma~2.1}, the mean $\mu$ and variance $\sigma^2$ of $\Theta$ satisfy
$$
\mu
=|Y(\bF_q)|\left( q^{\ell-N}
+O\left(q^{\ell-N-1}\right) \right),
\qquad \sigma^2
=O\left(q^{\ell-N}|Y(\bF_q)|\right).
$$
If $Y$ is geometrically irreducible, the Lang--Weil bound \cite{LW54}*{Theorem~1} gives
$$
    |Y(\bF_q)| = q^{\dim Y}
        + O\left(
            q^{\dim Y - \frac{1}{2}}
        \right).
$$
Substituting this into the preceding estimates, we obtain \begin{equation}
\label{eqn:mean-var}
\mu
=q^{\dim Y+\ell-N}
+O\left(
q^{\dim Y+\ell-N-\frac12}
\right), \qquad
\sigma^2
=O\left(q^{\dim Y+\ell-N}\right). \end{equation}

\begin{proof}[Proof of Theorem~\ref{thm:main}~(1)]
The inequality $D_q(Y, \ell)\geq 0$ holds by definition, so it remains to show that there exists a constant $c_1$, depending only on the fixed scheme $\mathscr{Y}$ and on $\ell$ such that
$$
   D_q(Y, \ell) \leq c_1q^{N - \dim Y - \ell}.
$$
Consider the set of pointless linear sections
$$
    \mathcal{U}_q\colonequals\left\{
        L\in\Gr(\bP^\ell,\bP^N)(\bF_q)
    \;\middle|\;
        (L\cap Y)(\bF_q) = \emptyset
    \right\}.
$$
For all sufficiently large $q$, Equation~\eqref{eqn:mean-var} gives
$\mu>0$. Since $\Theta=0$ on $\mathcal{U}_q$, Chebyshev's inequality
then gives
\begin{align*}
    D_q(Y,\ell)
    =\operatorname{Prob}(\Theta=0)
    &\leq\operatorname{Prob}(|\Theta-\mu|\geq\mu) \\
    &\leq\frac{\sigma^2}{\mu^2}
    = O\left(q^{N-\dim Y-\ell}\right).
\end{align*}
If $\sigma=0$, then $\Theta\equiv\mu>0$, so $\mathcal{U}_q$ is empty and the bound holds trivially. Increasing the constant to cover the finitely many smaller prime powers establishes the desired upper bound for every $q$. In particular, the condition $\dim Y + \ell > N$ implies $\lim_{q\to\infty}q^{N - \dim Y - \ell} = 0$, which further implies $\lim_{q\to\infty}D_q(Y, \ell) = 0$.
\end{proof}

\subsection{The abundant case}

Consider a geometrically irreducible variety $Y\subseteq\bP^N$ defined over~$\bF_q$, together with the incidence correspondence
$$
    \mathcal{X}\colonequals\left\{
        (y, L)\in Y\times \Gr(\bP^{\ell}, \bP^{N})
    \;\middle|\;
        y\in L
    \right\}.
$$
There are two natural projection maps
$$\xymatrix{
    & \mathcal{X}\ar[dl]_-{\pi_Y}\ar[dr]^-{\pi_{\Gr}} & \\
    Y && \Gr(\bP^{\ell}, \bP^{N}),
}$$
where the image of the second projection is
$$
    \mathcal{N}\colonequals\pi_{\Gr}(\mathcal{X}) = \left\{
        L\in\Gr(\bP^{\ell}, \bP^{N})
    \;\middle|\;
        L\cap Y\neq\emptyset
    \right\}.
$$
To clarify the (geometric) points of $ \mathcal{N}$, observe that
\begin{equation}\label{eq:description-of-N}
\mathcal{N}(\overline{\bF_q})=\left\{
L\in \Gr(\bP^\ell,\bP^N)(\overline{\bF_q}) \;\middle|\; L\cap Y_{\overline{\bF_q}}\neq\emptyset \right\}.
\end{equation}

\begin{lemma}
\label{lem:dim_incidence}
The incidence correspondence $\mathcal{X}$ is geometrically irreducible with
$$
    \dim\mathcal{X} = \dim\Gr(\bP^\ell,\bP^N) - (N - \dim Y - \ell).
$$
\end{lemma}
\begin{proof}
If $\ell=0$, then $\Gr(\bP^0,\bP^N)\cong\bP^N$, and $\mathcal{X}$ is the graph of the inclusion $Y\hookrightarrow\bP^N$. Since $\mathcal{X}\cong Y$, the assertion follows. For the remainder of the proof, we assume that $\ell\geq1$.

For any $y\in Y$, the fiber $\pi_Y^{-1}(y)$ consists of $\ell$-dimensional linear subspaces of $\bP^{N}$ containing $y$. Projecting from $y$ identifies such subspaces with $(\ell-1)$-dimensional linear subspaces of $\bP^{N-1}$, yielding an isomorphism
$$
    \pi_Y^{-1}(y)\cong\Gr(\bP^{\ell-1}, \bP^{N-1}).
$$
Hence, the fibers of $\pi_Y$ are geometrically irreducible and have the same dimension. Since~$Y$ is geometrically irreducible, it follows that $\mathcal{X}$ is geometrically irreducible by \cite{Sha13}*{Theorem~1.26}.

For the dimension, we compute
\begin{align*}
    \dim\mathcal{X} = \dim \pi_Y^{-1}(y) + \dim Y
    &= \dim\Gr(\bP^{\ell-1}, \bP^{N-1}) + \dim Y \\
    &= \ell(N-\ell)+\dim Y \\
    &= \ell(N-\ell)+(N-\ell)+\dim Y-(N-\ell) \\
    &= (\ell+1)(N-\ell)+\dim Y-(N-\ell) \\
    &= \dim \Gr(\bP^{\ell}, \bP^{N})-(N-\dim Y-\ell).
\end{align*}
This proves the statement.
\end{proof}

\begin{lemma}
\label{lem:img-in-Gr}
The projection image $\mathcal{N} = \pi_{\Gr}(\mathcal{X})$ is geometrically irreducible with
$$
    \dim\mathcal{N}\leq\dim\Gr(\bP^\ell,\bP^N) - (N - \dim Y - \ell). 
$$
\end{lemma}
\begin{proof}
The variety $\mathcal{N}$ is geometrically irreducible since $\mathcal{X}$ is geometrically irreducible by Lemma~\ref{lem:dim_incidence}. The upper bound on the dimension follows from the inequality
$
    \dim\mathcal{N}\leq\dim\mathcal{X}
$
together with Lemma~\ref{lem:dim_incidence}.
\end{proof}

\begin{proof}[Proof of Theorem~\ref{thm:main}~(2)]

By definition,
$$
    1-D_q(Y,\ell)
    =
    \frac{|\left\{
          L\in\Gr(\bP^\ell,\bP^N)(\bF_q)
      \;\middle|\;
          (L\cap Y)(\bF_q)\neq\emptyset
      \right\}|}
    {|\Gr(\bP^\ell,\bP^N)(\bF_q)|}.
$$
In view of \eqref{eq:description-of-N}, observe that $\left\{ L\in\Gr(\bP^\ell,\bP^N)(\bF_q) \;\middle|\; (L\cap Y)(\bF_q)\neq\emptyset \right\}$ is a subset of $\mathcal{N}(\bF_q)$. This inclusion may be strict, since an $\bF_q$-rational linear subspace can meet $Y$ geometrically without meeting $Y(\bF_q)$. In any event,
\begin{equation}
\label{eqn:03}
    1-D_q(Y,\ell)
    \leq
    \frac{|\mathcal{N}(\bF_q)|}
    {|\Gr(\bP^\ell,\bP^N)(\bF_q)|}.
\end{equation}

The variety $\mathcal{N}$ is geometrically irreducible by Lemma~\ref{lem:img-in-Gr}. By the bounded-complexity form of the Lang--Weil estimate \cite{CGH23}*{Theorem~4.6}, there is a constant $A$, independent of $q$, such that
$$
    \left|
        |\mathcal{N}(\bF_q)| - q^{\dim \mathcal{N}}
    \right|
    \leq A q^{\dim \mathcal{N}-1/2}.
$$
From this inequality, one can check that for $q\geq\max\{A^2, 1\}$,
\begin{equation}
\label{eqn:points-on-N}
    |\mathcal{N}(\bF_q)|\leq  q^{\dim \mathcal{N}}+Aq^{\dim \mathcal{N}-1/2}\leq 2q^{\dim \mathcal{N}}. 
\end{equation}

Next, we observe that 
\begin{equation}\label{eqn:gras-lower-bound}
|\Gr(\bP^\ell,\bP^N)(\bF_q)| = \prod_{j=0}^{\ell} \frac{q^{N+1-j}-1}{q^{\ell+1-j}-1}\geq q^{(\ell+1)(N-\ell)} = q^{\dim\Gr(\bP^\ell,\bP^N)}.
\end{equation}
Indeed, each factor is at least $q^{N-\ell}$, since $(q^{N+1-j}-1)-q^{N-\ell}(q^{\ell+1-j}-1)= q^{N-\ell}-1 \geq 0$.

Applying Lemma~\ref{lem:img-in-Gr} in combination with \eqref{eqn:03}, \eqref{eqn:points-on-N}, and \eqref{eqn:gras-lower-bound}, we obtain the following estimate for all sufficiently large $q$:
$$
    1 - D_q(Y, \ell)
    \leq\frac{
        2q^{\dim \Gr(\bP^{\ell}, \bP^{N})-(N-\dim Y-\ell)}
    }{
        q^{\dim \Gr(\bP^{\ell}, \bP^{N})}
    }
    = 2q^{\dim Y + \ell - N}.
$$
Increasing the constant to cover the finitely many smaller prime powers, we obtain a constant $c_2$ depending only on $\mathscr Y$ and $\ell$ such that, for every $q$,
$$
    1 - D_q(Y, \ell)\leq c_2q^{\dim Y + \ell - N}.
$$
Rearranging the terms gives
$$
    1 - c_2q^{\dim Y+\ell-N} \leq D_q(Y, \ell)\leq 1.
$$
If $\dim Y + \ell < N$, then $\lim_{q\to\infty}(1 - c_2q^{\dim Y+\ell-N}) = 1$, and thus $\lim_{q\to\infty} D_q(Y, \ell) = 1$.
\end{proof}

\section{Asymptotic density in the threshold case}
\label{sec:threshold}

Consider a geometrically irreducible variety $Y\subseteq\bP^N$ of degree~$\Delta$ over $\bF_q$, together with the incidence correspondence
$$\xymatrix{
    & \mathcal{X}\colonequals\left\{
        (y, L)\in Y\times \Gr(\bP^{\ell}, \bP^{N})
    \;\middle|\;
        y\in L
    \right\} \ar[dl]_-{\pi_Y}\ar[dr]^-{\pi_{\Gr}} & \\
    Y && \mathcal{G}\colonequals\Gr(\bP^{\ell}, \bP^{N}).
}$$
Assume that $\dim Y + \ell = N$. Then a general $\ell$-plane $L\in\mathcal{G}$ intersects $Y$ transversely in $\Delta$ points by Bertini's theorem. In particular, the map
$$
    \pi_{\Gr}\colon \mathcal{X}\longrightarrow\mathcal{G}
$$
is generically finite of degree~$\Delta$. Let $\mathcal G^\circ\subseteq\mathcal G$ be the locus of
$\ell$-planes whose intersection with $Y$ consists of $\Delta$
distinct geometric points, all smooth on $Y$. Observe that $\mathcal G^\circ$ is a nonempty Zariski open subset over which $\pi_{\Gr}$ is finite \'etale of degree~$\Delta$. In this setting, Theorem~\ref{thm:main}~(3) concerns the distribution of $\bF_q$-points $L\in\mathcal{G}(\bF_q)$ for which the fiber $\pi_{\Gr}^{-1}(L)=L\cap Y$ contains no $\bF_q$-points. 

\subsection{Chebotarev density theorem}

Before proceeding to the proof, we briefly review the background for the key ingredient. The main source for the following material is \cite{Ent21}*{Sections~3 and 4}. Let $\mathcal{G}^\circ$ be a geometrically irreducible quasi-projective variety, and let
$$
    f\colon\mathcal{X}^\circ\longrightarrow\mathcal{G}^\circ
$$
be a finite \'etale morphism of degree~$\Delta$. For a geometric point $L_0\in\mathcal{G}^\circ(\overline{\bF_q})$, the \'etale fundamental group $\pi^{\mathrm{et}}(\mathcal{G}^\circ, L_0)$ acts on the fiber $f^{-1}(L_0)$, giving rise to a homomorphism to the symmetric group~$S_\Delta$. The \emph{arithmetic monodromy group} $\Mon(\mathcal{X}^\circ/\mathcal{G}^\circ)$ is defined as the image of this homomorphism. This group is well defined up to conjugation since different choices of the base point $L_0$ yield conjugate subgroups of $S_\Delta$.

On the other hand, the \emph{geometric monodromy group} is defined by
$$
    \Mon^g(\mathcal{X}^\circ/\mathcal{G}^\circ)
    \colonequals
    \Mon\left(
        \mathcal{X}^\circ\times_{\bF_q}\overline{\bF_q}
        \big/
        \mathcal{G}^\circ\times_{\bF_q}\overline{\bF_q}
    \right).
$$
The two monodromy groups fit into a short exact sequence
$$
1\longrightarrow\Mon^g(\mathcal{X}^\circ/\mathcal{G}^\circ)
    \longrightarrow\Mon(\mathcal{X}^\circ/\mathcal{G}^\circ)
    \longrightarrow\Gal(\bF_{q^\nu}/\bF_q)
    \longrightarrow 1,
$$
where $\nu$ is the minimal positive integer such that
$$
    \Mon^g(\mathcal{X}^\circ/\mathcal{G}^\circ)
    = \Mon\left(
        \mathcal{X}^\circ\times_{\bF_q}\bF_{q^\nu}
        \big/
        \mathcal{G}^\circ\times_{\bF_q}\bF_{q^\nu}
    \right).
$$
In particular, the arithmetic and geometric monodromy groups may not coincide, even when $q$ is large. In our application, we will prove directly that the geometric monodromy group is $S_\Delta$. Since the arithmetic monodromy group is a subgroup of $S_\Delta$ containing the geometric monodromy group, it will follow that both groups are equal to $S_\Delta$.

For each point $L\in\mathcal{G}^\circ(\bF_q)$, the fiber $f^{-1}(L)\subseteq\mathcal{X}^\circ$ is a $0$-dimensional scheme over $\bF_q$ consisting of $\Delta$ geometric points. The action of the Frobenius map on this fiber determines a conjugacy class
$$
    \Frob_q(L)\subseteq\Mon(\mathcal{X}^\circ/\mathcal{G}^\circ).
$$
Note that $f^{-1}(L)$ contains no $\bF_q$-points if and only if the class $\Frob_q(L)$ consists of derangements. This idea reduces our problem to estimating the density of points $L\in\mathcal{G}^\circ(\bF_q)$ for which $\Frob_q(L)$ consists of derangements.

The key tool for our argument is the following version of the Chebotarev density theorem proved by Entin \cite{Ent21}. We state only the form relevant to our setting, namely the case in which the arithmetic and geometric monodromy groups coincide.

\begin{thm}[\cite{Ent21}*{Theorem~3}]
\label{thm:entin}
Let $f\colon\mathcal{X}^\circ\to\mathcal{G}^\circ$ be a finite \'etale morphism between quasiprojective varieties over $\bF_q$ with $\mathcal{G}^\circ$ smooth and geometrically irreducible. Suppose that 
$$
\Mon^g(\mathcal{X}^\circ/\mathcal{G}^\circ)
    = \Mon(\mathcal{X}^\circ/\mathcal{G}^\circ).
$$
Then for each conjugacy class $\mathcal{C}\subseteq\Mon(\mathcal{X}^\circ/\mathcal{G}^\circ)$, we have
$$
    \left|
        \left\{
            L\in\mathcal{G}^\circ(\bF_q)
            \;\middle|\;
            \Frob_q(L) = \mathcal{C}
        \right\}
    \right|
    = \frac{
        |\mathcal{C}|
    }{
        |\Mon(\mathcal{X}^\circ/\mathcal{G}^\circ)|
    } q^{\dim\mathcal{G}^\circ}\left(
        1 + O\left(q^{-1/2}\right)
    \right).
$$
\end{thm}

\begin{rmk}\label{rmk:complexity}
The implied constant in $O\left(q^{-1/2}\right)$ depends on the \emph{complexity} of the morphism~$f$. Roughly speaking, this complexity measures the ambient projective dimension and the degrees of the equations and nonvanishing conditions defining the graph of $f$; see \cite{Ent21}*{Section~4}. In our setting, $\mathscr Y$ is defined by fixed equations over
$\bZ$, and the incidence and transversality conditions have bounded
complexity as $q$ varies. The same holds for their images, the map
$f$, and the complement $\mathcal D=\mathcal G\setminus\mathcal G^\circ$.
It follows that the constants in the point-counting estimates and
in the Chebotarev theorem depend only on $\mathscr Y$ and $\ell$.

\end{rmk}

\subsection{Proof for the threshold case}

We retain the notation introduced at the beginning of this section and set $\mathcal X^\circ\colonequals\pi_{\Gr}^{-1}(\mathcal G^\circ)$. The restriction
$$
f\colonequals\pi_{\Gr}|_{\mathcal{X}^\circ}\colon\mathcal{X}^\circ\longrightarrow\mathcal{G}^\circ
$$
is therefore a finite \'etale morphism of degree~$\Delta$.

\begin{lemma}
\label{lem:full-monodromy-threshold}
If the variety $Y$ is quasireflexive, then
$\Mon^g(\mathcal X^\circ/\mathcal G^\circ)=\Mon(\mathcal X^\circ/\mathcal G^\circ) \cong S_\Delta$.
\end{lemma}
\begin{proof}
We first work over $\overline{\bF_q}$. Quasireflexivity implies
$\Delta\geq 2$ and $\ell\geq1$. Fix a geometric point
$L_0\in\mathcal{G}^\circ(\overline{\bF_q})$ and identify $\Mon^{g}(\mathcal{X}^\circ/\mathcal{G}^\circ)$ as a subgroup of $S_\Delta$ through its action on the geometric fiber $f^{-1}(L_0)$. Let
$$
    D\subseteq
    \mathcal{X}\times_\mathcal{G}\mathcal{X}
    \qquad\text{and}\qquad
    D'\subseteq Y\times Y
$$
denote the diagonals. The fiber of the morphism
$$
    \mathcal{X}\times_\mathcal{G}\mathcal{X}
    \setminus D
    \longrightarrow
    Y\times Y\setminus D'
$$
over $(y_1, y_2)\in Y\times Y\setminus D'$ consists of the points $((y_1, L),\, (y_2, L))$, where $L\in\mathcal{G} = \Gr(\bP^\ell, \bP^N)$ contains the line spanned by $y_1$ and $y_2$. This realizes $\mathcal{X}\times_\mathcal{G}\mathcal{X}\setminus D$ as a Grassmannian bundle over the irreducible variety $Y\times Y\setminus D'$, so it is itself irreducible. Its restriction
$$
    (\mathcal X^\circ\times_{\mathcal G^\circ}\mathcal X^\circ)
    \setminus D
$$
is an open subset, nonempty since $\Delta\geq 2$, and hence irreducible. It follows that
$\Mon^g(\mathcal{X}^\circ/\mathcal{G}^\circ)$ acts doubly transitively
on $f^{-1}(L_0)$ by \cite{BH86}*{Proposition~2~(ii)}.

Quasireflexivity gives $L\in\mathcal G$ whose intersection with $Y$ consists of one double point $y_0$ and $\Delta-2$ transverse points, all smooth on $Y$. The projection map $\pi_{\Gr}$ is finite near $L$ and \'etale at the transverse points. Since $\mathcal X\to Y$ is a Grassmannian bundle, $\mathcal X$ is smooth, hence formally irreducible, at $(y_0,L)$; the base $\mathcal G$ is also smooth. By \cite{BH86}*{Proposition~3}, as restated in
\cite{Ent21}*{Proposition~3.1}, the geometric monodromy group contains a transposition in every characteristic. Double transitivity now gives $\Mon^g(\mathcal X^\circ/\mathcal G^\circ)=S_\Delta$. Finally, the chain $$S_\Delta=\Mon^g(\mathcal{X}^\circ/\mathcal{G}^\circ) \subseteq \Mon(\mathcal{X}^\circ/\mathcal{G}^\circ) \subseteq S_\Delta$$
shows that equality must hold throughout, as desired.
\end{proof}

\begin{lemma}
\label{lem:derangements}
If the variety $Y$ is quasireflexive, then, as $q\to\infty$,
$$
    \left|\left\{
        L\in\mathcal{G}^\circ(\bF_q)
    \;\middle|\;
        {\Frob_q}\textup{ deranges }
        \pi_{\Gr}^{-1}(L)
    \right\}\right| = \left(
        \sum_{i=0}^{\Delta} \frac{(-1)^{i}}{i!}
    \right)
    q^{\dim\mathcal{G}}\left(
        1 + O\left(q^{-1/2}\right)
    \right).
$$
\end{lemma}
\begin{proof}
By Lemma~\ref{lem:full-monodromy-threshold}, the geometric and arithmetic monodromy groups are both the symmetric group $S_\Delta$. Applying Theorem~\ref{thm:entin}, we obtain
\begin{align*}
    \left|\left\{
        L\in\mathcal{G}^\circ(\bF_q)
    \;\middle|\;
        {\Frob_q}\textup{ deranges }
        \pi_{\Gr}^{-1}(L)
    \right\}\right|
    &= \sum_{\substack{
        \mathcal{C}\text{ consists of} \\
        \text{derangements}
    }}\frac{
        |\mathcal{C}|
    }{
        |S_{\Delta}|
    }
    q^{\dim\mathcal{G}}\left(
        1 + O\left(q^{-1/2}\right)
    \right) \\
    =& \frac{
        \left|\{\text{derangements in }S_{\Delta}\}\right|
    }{
        |S_{\Delta}|
    }
    q^{\dim\mathcal{G}}\left(
        1 + O\left(q^{-1/2}\right)
    \right) \\
    =& \left(
        \sum_{i=0}^{\Delta}\frac{(-1)^{i}}{i!}
    \right)
    q^{\dim\mathcal{G}}\left(
        1 + O\left(q^{-1/2}\right)
    \right),
\end{align*}
as claimed.
\end{proof}

\begin{proof}[Proof of Theorem~\ref{thm:main}~(3)]
By definition,
$$
    D_q(Y,\ell)
    = \frac{
        \left|
            \left\{
                L\in\mathcal{G}(\bF_q)
            \;\middle|\;
                (L\cap Y)(\bF_q)=\emptyset
            \right\}
        \right|
    }{
        |\mathcal{G}(\bF_q)|
    }.
$$
The condition $(L\cap Y)(\bF_q)=\emptyset$ holds if and only if the Frobenius endomorphism fixes no point on the fiber $\pi_{\Gr}^{-1}(L)$. Therefore, over the open subset $\mathcal{G}^\circ\subseteq\mathcal{G}$, Lemma~\ref{lem:derangements} yields
$$
    \left|
        \left\{
            L\in\mathcal{G}^\circ(\bF_q)
        \;\middle|\;
            (L\cap Y)(\bF_q)=\emptyset
        \right\}
    \right|
    =
    \left(
        \sum_{i=0}^{\Delta}\frac{(-1)^i}{i!}
    \right)
    q^{\dim\mathcal{G}}
    \left(
        1+O\left(q^{-1/2}\right)
    \right).
$$
The complement $\mathcal{D}\colonequals \mathcal{G}\setminus \mathcal{G}^\circ$ is a proper closed subset, so $\dim\mathcal{D}<\dim\mathcal{G}$. Thus, by the Lang--Weil bound \cite{LW54}*{Lemma~1},
$$
    |\mathcal{D}(\bF_q)|
    \leq O\left(q^{\dim\mathcal{G}-1}\right).
$$
Comparing the two estimates, the contribution from $\mathcal{D}(\bF_q)$ is negligible. On the other hand, the denominator satisfies
$
    |\mathcal{G}(\bF_q)| =
    q^{\dim\mathcal{G}}\left(1+O\left(q^{-1}\right)\right).
$
Therefore,
$$
    D_q(Y,\ell)
    = \sum_{i=0}^{\Delta}\frac{(-1)^i}{i!}
    + O\left(q^{-1/2}\right).
$$
In particular,
$$
    \lim_{q\to\infty}D_q(Y,\ell)
    = \sum_{i=0}^{\Delta}\frac{(-1)^i}{i!}.
$$
This completes the proof.
\end{proof}

\section{Quasireflexivity of determinantal varieties}
\label{sec:quasiref}

In this section, we prove that determinantal varieties are quasireflexive over fields of arbitrary characteristic. Throughout, we fix the following notation. Let $K$ be an algebraically closed field and fix integers 
$$
2\leq\delta\leq n\leq m.
$$
We consider the determinantal variety
$$
Y_\delta\subseteq\bP^N \qquad \text{where} \qquad N = mn - 1
$$
consisting of $m\times n$ matrices of rank less than~$\delta$. To simplify the notation, we denote its dimension and codimension by
\begin{align*}
    d &\colonequals \dim Y_\delta
    = N - (m - \delta + 1)(n - \delta + 1), \\
    c &\colonequals \operatorname{codim} Y_\delta
    = (m - \delta + 1)(n - \delta + 1).
\end{align*}
By Definition~\ref{def:quasireflex}, to show that $Y_\delta$ is quasireflexive, we need to find a $c$-dimensional linear subspace $L\subseteq\bP^N$ that intersects $Y_\delta$ in $\deg(Y_\delta) - 1$ points all lying in the smooth locus
$$
    Y_\delta^{\rm sm}\subseteq Y_\delta.
$$
We begin by treating the hypersurface case.

\begin{prop}
\label{prop:det-hyper}
If $Y_\delta$ is a hypersurface, that is, if $m = n = \delta$, then it is quasireflexive.
\end{prop}
\begin{proof}
In this setting, $Y_\delta\subseteq\bP^{N}=\mathbb{P}^{n^2-1}$ is a hypersurface of degree~$n$. It suffices to construct a line $L\subseteq\bP^N$ meeting the smooth locus $Y_\delta^{\rm sm}$ in exactly $n-1$ distinct points.

Choose $n-1$ distinct points $[a_i:b_i]\in\bP^1$, where $i = 1,\dots,n-1$. When $n=2$, choose in addition a point $[a_2:b_2]\neq[a_1:b_1]$. Consider the morphism
$$
    \varphi\colon\bP^1\longrightarrow\bP^N,
    \quad\varphi(s:t) = \begin{pmatrix}
        a_1 s + b_1 t & a_2 s + b_2 t & 0 & \cdots & 0 \\ 
        0 & a_1 s + b_1 t & 0 & \cdots & 0 \\ 
        0 & 0 & a_2 s + b_2 t & \cdots & 0 \\ 
        \vdots & \vdots & \vdots & \ddots & \vdots \\ 
        0 & 0 & 0 & \cdots & a_{n-1} s+ b_{n-1} t \\ 
    \end{pmatrix}.
$$
When $n=2$, this matrix is interpreted as
$$
\begin{pmatrix}
        a_1s + b_1t & a_2s + b_2t \\
        0 & a_1s+b_1t
\end{pmatrix}.
$$
Let $L\colonequals\varphi(\bP^1)$, which is a line in $\bP^N$. The scheme-theoretic intersection $L\cap Y_\delta$ is defined by the equation $\det(\varphi(s:t))=0$, namely
$$
(a_1 s + b_1 t)^2 (a_2 s + b_2 t) \cdots (a_{n-1}s + b_{n-1}t) = 0.
$$
When $n=2$, this equation reduces to $(a_1s+b_1t)^2=0$. The determinant equation above has a double root at $[b_1:-a_1]$ and $n-2$ simple roots at $[b_i:-a_i]$ for $2\leq i\leq n-1$.

For the matrix $\varphi(b_1:-a_1)$, the first two
diagonal entries vanish, while the $(1,2)$-entry equals
$a_2b_1 - a_1b_2\neq 0$, and all remaining diagonal
entries are nonzero. As the matrix has rank $n-1$, it represents a point in the smooth locus $Y_\delta^{\rm sm}$. At each remaining root $[b_i:-a_i]$, where $i\geq2$, exactly
one diagonal entry vanishes, while the leading $2\times2$ block
is invertible. Thus the matrix again has rank $n-1$, so every
remaining intersection point also lies in $Y_\delta^{\rm sm}$. Therefore, $L$ meets $Y_\delta^{\rm sm}$ in $n-1$ distinct points. This proves that $Y_\delta$ is quasireflexive.
\end{proof}

\subsection{Preliminary lemmas on tangent spaces}

The remainder of this section is devoted to the non-hypersurface case. Here, we collect properties of the tangent spaces of a determinantal variety that are needed for our proof.

For any point $P \in Y_\delta^{\rm sm}$, we write its embedded tangent space as
$
    T_P Y_\delta\subseteq\bP^N.
$
Recall that, as shown for example in \cite{Har92}*{Example~14.16},
$$
    T_PY_\delta = \left\{
        Q\in\bP^N
    \;\middle|\;
       Q(\ker P)
       \subseteq
       \operatorname{im}P
    \right\}.
$$
We will make use of the distinguished smooth point
$$
    P_0 =
    \begin{pmatrix}
        I_{\delta - 1} & 0 \\
        0 & 0
    \end{pmatrix}
    \in Y_\delta^{\rm sm},
$$
where $I_{\delta - 1}$ denotes the $(\delta - 1) \times (\delta - 1)$ identity matrix. Its tangent space $T_{P_0}Y_\delta$ consists of points represented by matrices of the form
\begin{equation}
\label{eqn:repre_TP0Y}
    \begin{pmatrix}
        \ast_{\delta-1,\,\delta-1}
        & \ast_{\delta-1,\,n-\delta+1} \\
        \ast_{m-\delta+1,\,\delta-1}
        & 0
    \end{pmatrix}
\end{equation}
where $\ast_{\alpha,\, \beta}$ denotes an arbitrary $\alpha\times\beta$ matrix.

We begin by analyzing the structure of the linear section $Y_\delta^{\rm sm}\cap T_{P_0}Y_\delta$. This section admits a stratification according to the dimension of the intersection
$
    \ker P_0\cap \ker P\subseteq K^{n}.
$
More precisely, we have
$$
    Y_\delta^{\rm sm}\cap T_{P_0}Y_\delta
    = \bigcup_{i=0}^{n-\delta+1}T_i
$$
where
$$
    T_i = \{
        P\in Y_\delta^{\rm sm}\cap T_{P_0}Y_\delta
    \mid
        \dim(\ker P_0\cap \ker P) = (n-\delta+1)-i
    \}.
$$
In particular, since $\ker P_0$ has dimension $n-\delta+1$, we have
$$
    T_0 = \{
        P\in Y_\delta^{\rm sm}\cap T_{P_0}Y_\delta
    \mid
        \ker P_0 = \ker P
    \}.
$$

\begin{lemma}
\label{lem:Ti}
The Zariski closure $\overline{T_0}\subseteq Y_\delta$ is a linear
space containing $P_0$.
The stratum $T_i$ is nonempty if and only if
$0\leq i\leq\min\{\delta-1,n-\delta+1\}$.
Each nonempty $T_i$ has positive codimension in $Y_\delta^{\rm sm}$.
If $Y_\delta$ is not a hypersurface, only $T_0$ can have
codimension one, and this occurs if and only if $\delta=n=2$.
\end{lemma}

\begin{proof}
The points of $T_0$ are precisely the rank-$(\delta-1)$ matrices
in the linear space
$$
    \left\{
        P\in\bP^N
    \;\middle|\;
        P(\ker P_0)=0
    \right\}.
$$
This linear space is contained in $Y_\delta$ and contains $P_0$.
Since $T_0$ is a nonempty open subset of it, $\overline{T_0}$
is the whole linear space.

To compute the dimension of $T_i$, we first decompose
$$
    K^n\cong\ker P_0\oplus K^{\delta-1}.
$$
For each $P\in T_i$, the restrictions of a linear map
$\widetilde P\colon K^n\longrightarrow K^m$ representing it
to the two summands behave as follows:
\begin{itemize}
\item Since $P\in T_{P_0}Y_\delta$, the restriction to
$\ker P_0$ defines a linear map
$$
    A\colon\ker P_0\cong K^{n-\delta+1}
    \longrightarrow\operatorname{im}P_0\cong K^{\delta-1}
$$
of rank $i$.
\item The restriction to the second summand defines a linear map
$$
    B\colon K^{\delta-1}\longrightarrow K^m.
$$
Its composition with the quotient map gives
$$
    \overline B\colon K^{\delta-1}
    \longrightarrow K^m/\operatorname{im}A.
$$
Since
$\operatorname{rank}\widetilde P
=i+\operatorname{rank}\overline B=\delta-1$,
the map $\overline B$ has rank $\delta-1-i$.
\end{itemize}
Conversely, any pair of restrictions satisfying these conditions
determines a point of $T_i$. Such choices exist precisely when
$$
    0\leq i\leq\min\{\delta-1,n-\delta+1\}.
$$

The rank-$i$ maps $A$ form a variety of dimension $i(n-i)$.
For fixed $A$, the maps $\overline B$ of rank $\delta-1-i$
form a variety of dimension $(\delta-1-i)\bigl((\delta-1)+(m-i)-(\delta-1-i)\bigr) =m(\delta-1-i)$. Each such map has an affine space of lifts $B$ of dimension
$$
    \dim\operatorname{Hom}
    \left(K^{\delta-1},\operatorname{im}A\right)
    =i(\delta-1).
$$
Subtracting one for projectivization, we obtain
$$
    \dim T_i
    =i(n-i)+m(\delta-1-i)+i(\delta-1)-1.
$$

To determine the codimension of $T_i$ in $Y_\delta^{\rm sm}$,
we compute
\begin{align*}
    \dim Y_\delta-\dim T_i
    &= (\delta-1)(m+n-\delta+1)-\bigl(i(n-i)+m(\delta-1-i)+i(\delta-1)\bigr)\\
    &= (\delta-1-i)(n-\delta+1-i)+i(m-\delta+1).
\end{align*}
For $i=0$, this is $(\delta-1)(n-\delta+1)$, which is positive
and equals one if and only if $\delta=n=2$.
For $i>0$, the codimension is at least $i(m-\delta+1)>0$.
If it equals one, then $i=1$ and $m=\delta$; since
$\delta\leq n\leq m$, this forces $m=n=\delta$.
Thus, outside the hypersurface case, only $T_0$ can have
codimension one. This completes the proof.
\end{proof}

\begin{lemma}
\label{lem:P0-notin-TPY}
For a general $P\in Y_\delta^{\rm sm}$, the tangent space $T_PY_\delta$ does not contain $P_0$.
\end{lemma}
\begin{proof}
Since $Y_\delta$ is irreducible and the locus of $P\in Y_\delta^{\rm sm}$ for which $T_PY_\delta$ contains $P_0$ is Zariski closed, it suffices to find a point $P$ such that $T_PY_\delta$ does not contain $P_0$. Let $P\in Y_\delta^{\rm sm}$ be the point represented by the matrix whose $(i, i)$-entries are equal to $1$ for $i = 2,\dots,\delta$, and zero elsewhere. Let $e_i$ denote the $i$th standard basis vector in $K^n$ or $K^m$, as appropriate. Then
$$
    e_1\in\ker P
    \quad\text{and}\quad
    \operatorname{im}P
    = \operatorname{span}(e_2,\dots,e_\delta).
$$
Since $P_0(e_1) = e_1$, we have
$$
    P_0(\ker P)
    \not\subseteq
    \operatorname{im}P,
$$
and hence $P_0\notin T_PY_\delta$. This gives an example of $P$ with the desired property.
\end{proof}

\begin{lemma}
\label{lem:TP0Y-int-TPY}
If $Y_\delta$ is not a hypersurface, then for a general point $P\in Y_\delta^{\rm sm}$, we have
$$
    \dim\left(
         T_{P_0}Y_\delta\cap T_PY_\delta
    \right)
    \leq d - 2.
$$
\end{lemma}
\begin{proof}
Since $Y_\delta$ is irreducible and the function
$
    P\longmapsto
    \dim\left(
        T_{P_0}Y_\delta\cap T_PY_\delta
    \right)
$
is upper semicontinuous on $Y_\delta^{\rm sm}$, it suffices to find a point $P$ with
$$
    \dim\left(
         T_{P_0}Y_\delta\cap T_PY_\delta
    \right)
    \leq d - 2,
$$
or equivalently, such that $T_{P_0}Y_\delta\cap T_PY_\delta$ has codimension at least $2$ in $T_{P_0}Y_\delta$. Consider
$$
    P = \begin{pmatrix}
        0 & 0 \\
        0 & I_{\delta-1}
    \end{pmatrix}\in Y_\delta^{\rm sm}.
$$
Its tangent space $T_PY_\delta$ consists of points represented by matrices of the form
$$
    \begin{pmatrix}
        0 & \ast_{m-\delta+1,\,\delta-1} \\
        \ast_{\delta-1,\,n-\delta+1}
        & \ast_{\delta-1,\,\delta-1}
    \end{pmatrix}.
$$
Comparing this with \eqref{eqn:repre_TP0Y}, we see that the intersection $T_{P_0}Y_\delta\cap T_PY_\delta$ consists of points represented by $m\times n$ matrices with $(m-\delta+1)\times(n-\delta+1)$ zero matrices as the top-left and bottom-right blocks.

\begin{itemize}
\item Suppose that 
$$
    m-\delta+1\leq \delta-1
    \qquad\text{or}\qquad
    n-\delta+1\leq \delta-1.
$$
In this case, the two zero blocks do not overlap. Thus the linear subspace
$$
    T_{P_0}Y_\delta\cap T_PY_\delta
    \subseteq T_{P_0}Y_\delta
$$
is defined by $(m-\delta+1)(n-\delta+1)$ independent linear conditions. Hence
$$
    \operatorname{codim}_{\,T_{P_0}Y_\delta}\left(
        T_{P_0}Y_\delta\cap T_PY_\delta
    \right) = (m-\delta+1)(n-\delta+1).
$$
Note that this quantity equals $1$ if and only if $m = n = \delta$.

\item Suppose that 
$$
    m-\delta+1 > \delta-1
    \qquad\text{and}\qquad
    n-\delta+1 > \delta-1.
$$
In this case, the two zero blocks overlap in a matrix of size
$$
    (m-2\delta+2)\times(n-2\delta+2).
$$
Consequently, the linear subspace
$
    T_{P_0}Y_\delta\cap T_PY_\delta
    \subseteq T_{P_0}Y_\delta
$
has codimension
\begin{align*}
    &(m-\delta+1)(n-\delta+1) - (m-2\delta+2)(n-2\delta+2) \\
    &= (\delta - 1)(m + n - 3\delta + 3).
\end{align*}
The inequalities above give $m\geq 2\delta-1$ and $n \geq 2\delta-1$, so this codimension is at least $(\delta-1)(\delta+1)\geq 3$.
\end{itemize}

In both cases, the intersection $T_{P_0}Y_\delta\cap T_PY_\delta$ has codimension at least $2$ in $T_{P_0}Y_\delta$ under the assumption that $Y_\delta$ is not a hypersurface. This completes the proof.
\end{proof}

\subsection{Finite linear sections with a tangency}

A $c$-dimensional linear subspace $L\subseteq\bP^N$ containing
$P_0$ is tangent to $Y_\delta$ at $P_0$ if and only if
$L\cap T_{P_0}Y_\delta$ contains a line through $P_0$.
The $c$-dimensional linear subspaces with this property form the variety
$$
    X_0\colonequals\left\{
        L\in\Gr(\bP^c, \bP^N)
    \;\middle|\;
        L\ni P_0,\;
        \dim(L\cap T_{P_0}Y_\delta) \geq 1
    \right\}.
$$
Our goal is to find an $L \in X_0$ that is tangent to $Y_\delta$ at $P_0$ with multiplicity two and transverse at all other intersection points. The proof proceeds in three steps:
\begin{enumerate}[label=\textup{(\arabic*)}]
    \item There exists $L\in X_0$ that is tangent to $Y_\delta$ at $P_0$ with multiplicity two.
    \item The locus of $L\in X_0$ that meet the singular locus $Y_{\delta - 1}\subseteq Y_\delta$ has codimension $\geq 2$.
    \item The locus of $L\in X_0$ that are tangent to $Y_\delta^{\rm sm}$ at a second point has codimension $\geq 1$.
\end{enumerate}

\begin{lemma}
\label{lem:L-tan-to-P0}
The variety $X_0$ is irreducible with $\dim X_0 = cd-1$.
\end{lemma}
\begin{proof}
The lines in $T_{P_0}Y_\delta$ passing through $P_0$ form the variety
$$
    \mathcal{B}\colonequals\left\{
        \ell\in\Gr(\bP^1, T_{P_0}Y_\delta)
    \;\middle|\;
        P_0\in\ell
    \right\}
    \cong\bP^{d - 1}.
$$
Consider the incidence correspondence
$$\xymatrix{
    & I_0\colonequals\left\{
        (L, \ell)\in X_0\times\mathcal{B}
    \;\middle|\;
        L\supseteq\ell
    \right\}\ar[dl]_-{\pi_1}\ar[dr]^-{\pi_2} & \\
    X_0 & & \mathcal{B}.
}$$
For each $\ell\in\mathcal{B}$, the fiber $\pi_2^{-1}(\ell)\subseteq I_0$ consists of all pairs $(L, \ell)$ such that $L\in\Gr(\bP^c, \bP^N)$ contains~$\ell$. This gives $I_0$ the structure of a bundle over $\mathcal{B}$ with fiber $\Gr(\bP^{c-2}, \bP^{N-2})$. It follows that $I_0$ is irreducible with
$$
    \dim I_0 = \dim\Gr(\bP^{c-2}, \bP^{N-2}) + \dim\mathcal{B}
    = (c - 1)d + (d - 1)
    = cd - 1.
$$
The projection $\pi_1$ is surjective, so $X_0$ is irreducible.
Fix a line $\ell\in\mathcal B$. A general $c$-plane $L$
containing $\ell$ satisfies $L\cap T_{P_0}Y_\delta=\ell$. The fiber of $\pi_1$ over such an $L$ consists of the single
point $(L,\ell)$. Thus $\pi_1$ has zero-dimensional general
fibers, and $\dim X_0=\dim I_0=cd-1$.
\end{proof}

\begin{lemma}
\label{lem:double-not-triple}
There exists $L_0\in X_0$ that is tangent to $Y_\delta$ at $P_0$ with multiplicity two.
\end{lemma}

\begin{proof}
We work in the affine chart of $\mathbb{P}^N$ consisting of $m \times n$ matrices whose $(1,1)$-entry is equal to $1$. Let $L_0\subseteq\bP^N$ be the projective linear subspace given
in this chart by
$$
    M\colonequals\begin{pmatrix}
        I_{\delta - 1} & A \\
        B & C
    \end{pmatrix}
$$
where
\begin{itemize}
\item $A = (a_{ij})$ has $a_{ij} = 0$ for all $(i,j) \neq (1,1)$,
\item $B = (b_{ij})$ has $b_{11} = a_{11}$ and $b_{ij} = 0$ for all $(i,j) \neq (1,1)$,
\item $C = (c_{ij})$ has $c_{11} = 0$.
\end{itemize}
The $c$ free coordinates give $\dim L_0=c$; setting $C=0$ gives a line through $P_0$ in $T_{P_0}Y_\delta$, so $L_0\in X_0$. By the standard block row-reduction,
$$
    \mathrm{rank}(M)
    = \mathrm{rank}\begin{pmatrix}
        I_{\delta - 1} & A \\
        0 & C - BA
    \end{pmatrix}
    = (\delta - 1) + \mathrm{rank}(C - BA).
$$
The same row reduction shows that the $\delta\times\delta$ minors generate the ideal of the entries of $C-BA$. Thus $L_0\cap Y_\delta$ is defined in this chart by
$$
    b_{11}a_{11}=a_{11}^{2}=0
    \qquad\text{and}\qquad
    c_{ij}=0 \quad \text{for all} \quad (i, j).
$$
Thus, $P_0$ is isolated in $L_0\cap Y_\delta$ and has intersection multiplicity two, as required.
\end{proof}

\begin{lemma}
\label{lem:multiplicity-2-open}
There exists a nonempty Zariski open subset $U\subseteq X_0$ such that every $L\in U$ is tangent to $Y_\delta$ at $P_0$ with multiplicity two.
\end{lemma}

\begin{proof}
Choose $L_0$ as in Lemma~\ref{lem:double-not-triple}. Fix an affine chart $\bA^N\subseteq\bP^N$ containing $P_0$, with coordinates $z=(z_1,\ldots,z_N)$. Working in an affine chart of the Grassmannian containing $L_0$, we may choose a Zariski open neighborhood
$W\subseteq X_0$ of $L_0$ and defining equations for the subspaces $L\in W$ of the form
$$
f_{L,i}(z)=c_{i0}(L)+\sum_{j=1}^{N}c_{ij}(L)z_j=0, \qquad \text{ for } i=1,\ldots,d,
$$
where each coefficient $c_{ij}$ is a regular function on $W$.

After re-indexing, we may assume that $f_{L_0,1},\ldots,f_{L_0,d-1}$ have linearly independent differentials on $T_{P_0}Y_\delta$. This is possible since $L_0\cap Y_\delta$ has length two at $P_0$, so its tangent space $L_0\cap T_{P_0}Y_\delta$ is one-dimensional, and the $d$ differentials $df_{L_0,i}|_{T_{P_0}Y_\delta}$ span a space of dimension exactly $d-1$. Complete $f_{L_0,1},\ldots,f_{L_0,d-1}$ to a regular system of parameters on $Y_\delta$ at $P_0$ by a function $t$. After shrinking $W$, for every $L\in W$ the common zero locus of $f_{L,1},\ldots,f_{L,d-1}$ on $Y_\delta$ in the chosen affine chart is smooth of dimension one at $P_0$. We may therefore choose a sufficiently small open neighborhood $C_L$ of $P_0$ such that $C_L$ is a smooth curve and $t_L\colonequals t|_{C_L}$ is a local parameter at $P_0$.

Since $P_0\in L$, the restriction $f_{L,d}|_{C_L}$ vanishes at $P_0$, so its expansion in the local parameter $t_L$ has no constant term. To see that the linear term also vanishes, observe that the first $d-1$ differentials $df_{L,1}(P_0),\ldots,df_{L,d-1}(P_0)$ have a one-dimensional common kernel on $T_{P_0}Y_\delta$, namely the tangent space to $C_L$ at $P_0$. Since $L$ is tangent to $Y_\delta$ at $P_0$, the common
kernel of all $d$ differentials on $T_{P_0}Y_\delta$ is nonzero. This kernel is contained in the one-dimensional tangent space to $C_L$ at $P_0$, so the two spaces coincide. In particular, $df_{L,d}(P_0)$ vanishes on the tangent space to $C_L$, which means that the coefficient of $t_L$ is zero. Consequently, for some scalar $a(L)$, we can express
\begin{equation}
\label{eqn:fLd=at2}
    f_{L,d}|_{C_L} \equiv a(L)t_L^2 \pmod{t_L^3}.
\end{equation}

We next prove that $a(L)$ is regular. Set $x_i \colonequals f_{L_0,i}|_{Y_\delta}$ for $i=1,\ldots,d-1$ and write $x=(x_1,\ldots,x_{d-1})^T$. The functions $x_1,\ldots,x_{d-1},t$ are regular near $P_0$ on $Y_\delta$ and are fixed as $L$ varies in $W$. Write $F_L=(f_{L,1},\ldots,f_{L,d-1})^T$ and expand its restriction to $Y_\delta$ through degree two in these fixed local parameters:
$$
    F_L|_{Y_\delta}(x,t)
    = J(L)x + b(L)t + Q_L(x,t) + O(3),
$$
where each entry of $Q_L$ is homogeneous quadratic and $O(3)$ denotes terms of total degree at least
three in $x,t$. All coefficients are regular in $L$, and $J(L_0)$ is the identity matrix, so we may assume that $\det J(L)\neq 0$. 

Since each $x_i$ vanishes at $P_0$ and $t_L$ is a
local parameter on $C_L$, we may write
$$
    x_i|_{C_L}
    \equiv u_i(L)t_L+v_i(L)t_L^2 \pmod{t_L^3},
    \qquad \text{ for } i=1,\ldots,d-1.
$$
Put
$$
    u(L)=(u_1(L),\ldots,u_{d-1}(L))^T,
    \qquad
    v(L)=(v_1(L),\ldots,v_{d-1}(L))^T.
$$
To compute the quadratic contribution, use the
homogeneity of $Q_L$:
$$
    Q_L(u(L)t_L,t_L)
    =Q_L(t_Lu(L),t_L\cdot1)
    =t_L^2Q_L(u(L),1).
$$
Terms involving $v(L)t_L^2$ inside $Q_L$ have degree at least three in $t_L$. Substituting the expansions into $F_L|_{C_L}=0$ and comparing coefficients of $t_L$ and $t_L^2$ gives
$$
    J(L)u(L)=-b(L),
    \qquad
    J(L)v(L)=-Q_L(u(L),1).
$$
After possibly shrinking $W$, the entries of the inverse matrix $J(L)^{-1}$ are regular on a neighborhood of $L_0$ in $W$. Then the displayed equations above show that $u(L)$ and $v(L)$ are regular in $L$.

Similarly, expand $f_{L,d}|_{Y_\delta}$ in these parameters:
$$
    f_{L,d}|_{Y_\delta}(x, t)
    = \sum_{i=1}^{d-1}\lambda_i(L)x_i
        + \beta(L)t + R_L(x,t) + O(3),
$$
where $R_L$ is homogeneous quadratic and all coefficients are regular in $L$. Substituting the expansions of $x_i|_{C_L}$ into \eqref{eqn:fLd=at2} and comparing the coefficients of $t_L^2$, we get
$$
    a(L)=\sum_{i=1}^{d-1}\lambda_i(L)v_i(L)
    +R_L(u(L),1),
$$
which is regular in $L$ since $u(L)$, $v(L)$, and the coefficients of this expansion are regular.

Since $a(L_0)\neq0$, the nonvanishing locus of $a$ is a nonempty open subset $U\subseteq X_0$. For every $L\in U$, the point $P_0$ is an isolated double point of $L\cap Y_\delta$.
\end{proof}

\begin{lemma}
\label{lem:meet-sing}
The locus $X_s\subseteq X_0$ of linear subspaces $L\in X_0$ that meet the singular locus $Y_{\delta - 1}\subseteq Y_\delta$ is closed and has dimension at most $cd - 3$.
\end{lemma}
\begin{proof}
If $\delta=2$, then $Y_1$ is empty and there is nothing to prove.
Suppose $\delta\geq3$. Consider the closed incidence correspondence
$$\xymatrix{
    & I_s\colonequals\left\{
        (L, P)\in X_0\times Y_{\delta-1}
    \;\middle|\;
        L\ni P
    \right\}\ar[dl]_-{\pi_1}\ar[dr]^-{\pi_2} & \\
    X_0 & & Y_{\delta-1}.
}$$
Since $Y_{\delta-1}$ is projective, $\pi_1$ is proper and its
image $X_s$ is closed.
For each singular point $P\in Y_{\delta-1}$, every $L\in X_s$ containing $P$ must also contain the line spanned by $P_0$ and $P$. If nonempty, the fiber $\pi_2^{-1}(P)$ can be identified with a subset of $\Gr(\bP^{c-2}, \bP^{N-2})$. Therefore,
\begin{align*}
    \dim I_s
    &\leq \dim Y_{\delta-1} + \dim\Gr(\bP^{c-2}, \bP^{N-2}) \\
    &= (N - (m - \delta + 2)(n - \delta + 2)) + (c - 1)(N - c) \\
    &= (d - (m + n - 2\delta + 3)) + (c - 1)d \\
    &= cd - (m + n - 2\delta + 3) \\
    &\leq cd - 3.
\end{align*}
Since $X_s=\pi_1(I_s)$, we have $\dim X_s\leq \dim I_s\leq cd - 3$.
\end{proof}

\subsection{Finite linear sections with extra tangency}

Our next goal is to bound the dimension of the locus $X_t\subseteq X_0$ defined by
$$
    X_t\colonequals\left\{
        L\in\Gr(\bP^c, \bP^N)
    \;\middle|\;
        L\text{ is tangent at }P_0\text{ and some }
        P\in Y_\delta^{\rm sm}\setminus \{P_0\}
    \right\}.
$$
This locus is constructible: it is the image under
$(L,P)\mapsto L$ of the incidence defined by
$P\in L$ and $\dim(L\cap T_PY_\delta)\geq1$ in
$X_0\times(Y_\delta^{\rm sm}\setminus\{P_0\})$.

\begin{lemma}
\label{lem:L-tan-at-tan}
Suppose that $Y_\delta$ is not a hypersurface. Then the locus $X_t'\subseteq X_t$ defined by
$$
    X_t'\colonequals\left\{
        L\in X_t
    \;\middle|\;
        L\text{ is tangent at some }
        P\in
        (Y_\delta^{\rm sm}\cap T_{P_0}Y_\delta)
        \setminus\{P_0\}
    \right\}
$$
has dimension at most $cd-2$.
\end{lemma}

\begin{proof}
For each $P \in Y_\delta^{\mathrm{sm}} \setminus \{P_0\}$, the subspaces $L\in\Gr(\bP^c, \bP^N)$ containing both $P_0$ and $P$ are precisely those that contain the line spanned by $P_0$ and $P$. Such subspaces are parametrized by $\Gr(\bP^{c-2}, \bP^{N-2})$. By Lemma~\ref{lem:Ti}, the intersection
$
    Y_\delta^{\rm sm}\cap T_{P_0}Y_\delta
$
decomposes into the strata $T_i$. Since $\overline{T_0}$ is a
linear space containing $P_0$, the lines spanned by $P_0$ and
points of $T_0\setminus\{P_0\}$ form a family of dimension
$\dim T_0-1\leq d-2$. For each nonempty $T_i$ with $i>0$,
we have $\dim T_i\leq d-2$, so the lines spanned by $P_0$
and points of $T_i$ also form a family of dimension at most
$d-2$. Consequently,
\begin{align*}
    \dim X_t'
    &\leq \dim\Gr(\bP^{c-2}, \bP^{N-2}) + (d-2) \\
    &= (c-1)d + d-2 \\
    &= cd-2.
\end{align*}
This yields the desired bound.
\end{proof}

It remains to analyze the complement
$$
    X_t^\circ\colonequals X_t\setminus X_t'.
$$
For each $L\in X_t^\circ$ and each second tangency point
$P\in Y_\delta^{\rm sm}\setminus\{P_0\}$, there exist lines
$\ell_0\subseteq T_{P_0}Y_\delta$ and
$\ell\subseteq T_PY_\delta$ such that
\begin{equation}
\label{eqn:two-tangency}
    P_0\in\ell_0\subseteq L\cap T_{P_0}Y_\delta
    \qquad\text{and}\qquad
    P\in\ell\subseteq L\cap T_PY_\delta.
\end{equation}
In view of this, we define
\begin{align*}
    \mathcal{B}_0
    &\colonequals\left\{
        \ell_0\in\Gr(\bP^1, T_{P_0}Y_\delta)
    \;\middle|\;
        P_0\in\ell_0
    \right\}
    \cong\bP^{d - 1}, \\
    \mathcal{B}_1
    &\colonequals\left\{
        (P,\ell)\in
        (Y_\delta^{\rm sm}\setminus\{P_0\})
        \times\Gr(\bP^1,\bP^N)
    \;\middle|\;
        P\in\ell\subseteq T_PY_\delta
    \right\}.
\end{align*}
The projection
$\mathcal{B}_1\to Y_\delta^{\rm sm}\setminus\{P_0\}$
is a projective bundle with fiber $\bP^{d-1}$, so
$\dim\mathcal{B}_1=2d-1$. Consider the incidence correspondence
$$
    I_t\subseteq
    X_t^\circ\times\mathcal{B}_0\times\mathcal{B}_1,
$$
consisting of quadruples $(L, \ell_0, (P, \ell))$ satisfying \eqref{eqn:two-tangency}. Each such quadruple satisfies:
\begin{itemize}
    \item The lines $\ell_0$ and $\ell$ are distinct.
    \item If $\ell_0$ and $\ell$ intersect, the intersection point is different from $P$.
\end{itemize}
Indeed, if either of these conditions fails, then $P\in\ell_0\subseteq T_{P_0}Y_\delta$, contradicting the assumption that $L\in X_t^\circ$. Hence $I_t$ decomposes as a disjoint union
$$
    I_t = I_t'\cup I_t'',
$$
where
\begin{align*}
    I_t'\ &= \{
        (L, \ell_0, (P, \ell))\in I_t
    \mid
        \ell_0\cap\ell = \emptyset
    \}, \\
    I_t''\, &= \{
        (L, \ell_0, (P, \ell))\in I_t
    \mid
        \ell_0\cap\ell = \{Q\}\text{ for some }Q\neq P
    \}.
\end{align*}

\begin{lemma}
\label{lem:ext-tan_disjoint}
$\dim I_t'\leq cd - 2$.
\end{lemma}

\begin{proof}
If $c<3$, then $I_t'=\emptyset$. Suppose $c\geq3$.
Consider the projection map
$$
    \pi'\colon I_t'\longrightarrow
    \mathcal{B}_0\times\mathcal{B}_1
    : (L, \ell_0, (P, \ell))\longmapsto (\ell_0,(P,\ell)).
$$
For each quadruple $(L, \ell_0, (P, \ell))\in I_t'$, the condition
$\ell_0\cap\ell = \emptyset$ implies that their span
$$
    \Pi\colonequals\langle\ell_0, \ell\rangle\subseteq\bP^N
$$
is a $3$-plane. Any $L$ occurring in such a quadruple must
contain $\Pi$. Hence, over a point
$$
    (\ell_0,(P,\ell))\in\mathcal{B}_0\times\mathcal{B}_1,
$$
the fiber $(\pi')^{-1}(\ell_0,(P,\ell))$ is contained in the
Grassmannian of $c$-planes containing $\Pi$, which has
dimension $(c-3)(N-c)$. It follows that
\begin{align*}
    \dim I_t'
    &\leq \dim\mathcal{B}_0+\dim\mathcal{B}_1
        +(c-3)(N-c) \\
    &= (d-1)+(2d-1)+(c-3)(N-c) \\
    &= 3d-2+(c-3)d \\
    &= cd-2.
\end{align*}
This proves the desired upper bound.
\end{proof}

\begin{lemma}
\label{lem:ext-tan_a-point}
$\dim I_t''\leq cd - 2$, under the assumption that $Y_\delta$ is not a hypersurface.
\end{lemma}
\begin{proof}
Assigning to each $(L, \ell_0, (P, \ell))\in I_t''$ the intersection point $\ell_0\cap\ell$ defines a morphism
$$
    \iota\colon I_t''\longrightarrow\bP^N
    : (L, \ell_0, (P, \ell))\longmapsto\ell_0\cap\ell.
$$
Let $J_0$ denote the fiber over $P_0$ and $J_1$ its complement. That is,
$$
    J_0 \colonequals \iota^{-1}(P_0), \qquad
    J_1 \colonequals I_t''\setminus J_0.
$$

We first bound the dimension of $J_0$. Fix a point $P\in Y_\delta^{\rm sm}$ in the image of the projection
$$
    \pi\colon J_0\longrightarrow Y_\delta^{\rm sm}
    : (L, \ell_0, (P, \ell))\longmapsto P.
$$
The defining condition of $I_t''$ implies $P_0 = \ell_0\cap\ell\neq P$. Thus, for any $(L, \ell_0, (P, \ell))\in\pi^{-1}(P)$,
\begin{itemize}
\item $\ell$ is uniquely determined as the line through $P_0$ and $P$,
\item $\ell_0$ is a line in $T_{P_0}Y_\delta$ passing through $P_0$, and
\item $L$ is a $c$-dimensional linear subspace containing the plane spanned by $\ell_0$ and $\ell$.
\end{itemize}
The $c$-planes containing a fixed plane form a family of dimension $(c-2)d$; for $c=2$, the only member is the plane itself. Thus, the above conditions imply that
$$
\dim\pi^{-1}(P)\leq (c-2)d+(d-1).
$$
Moreover, we have
$
    P_0\in\ell\subseteq T_PY_\delta,
$
so in particular $T_PY_\delta$ contains $P_0$. By Lemma~\ref{lem:P0-notin-TPY}, the locus of such points $P\in Y_\delta^{\rm sm}$ has dimension at most $d - 1$. Combining this with the bound for each nonempty fiber, we obtain
$$
    \dim J_0\leq \bigl((c-2)d+(d-1)\bigr)+(d-1) 
    = cd - 2.
$$

Next, we bound the dimension of $J_1$. Consider the projection
$$
    \rho\colon J_1\longrightarrow Y_\delta^{\rm sm}
    : (L, \ell_0, (P, \ell))\longmapsto P.
$$
For each $P\in Y_\delta^{\rm sm}$, the restriction of $\iota$ to the fiber $\rho^{-1}(P)\subseteq J_1$ defines a map
$$
    \iota_P\colonequals\iota|_{\rho^{-1}(P)}
    \colon
    \rho^{-1}(P)\longrightarrow
    T_{P_0}Y_\delta\cap T_PY_\delta\cong\bP^e.
$$
By the defining conditions of $I_t''$ and $J_1$, we have
$$
    P_0\neq\ell_0\cap\ell\neq P.
$$
Thus, for each $Q\in\operatorname{im}(\iota_P)$, the fiber $\iota_P^{-1}(Q)$ consists of $(L, \ell_0, (P, \ell))\in\rho^{-1}(P)$ such that 
\begin{itemize}
\item $\ell_0$ is uniquely determined as the line through $P_0$ and $Q$,
\item $\ell$ is uniquely determined as the line through $P$ and $Q$, and
\item $L$ is a $c$-dimensional linear subspace containing the plane spanned by $\ell_0$ and $\ell$.
\end{itemize}
These imply that each nonempty fiber satisfies
$$
    \dim\rho^{-1}(P)\leq(c-2)d+e.
$$
For $P\in\operatorname{im}(\rho)$, the definition of $X_t^\circ$
gives $P\notin T_{P_0}Y_\delta$. Since $P\in T_PY_\delta$,
the two tangent spaces are distinct, so $e\leq d-1$.

Consider the locus
$$
    Z\colonequals\left\{
        P\in Y_\delta^{\rm sm}
    \;\middle|\;
        \dim\left(T_{P_0}Y_\delta\cap T_PY_\delta\right)
        \geq d-1
    \right\}.
$$
Since $Y_\delta$ is not a hypersurface,
Lemma~\ref{lem:TP0Y-int-TPY} shows that $Z$ is a proper subset of $Y_\delta^{\rm sm}$. It is closed by upper semicontinuity, so $\dim Z\leq d-1$. Consequently,
\begin{align*}
    \dim\rho^{-1}(Y_\delta^{\rm sm}\setminus Z)
    &\leq d+(c-2)d+(d-2)=cd-2,\\
    \dim\rho^{-1}(Z)
    &\leq(d-1)+(c-2)d+(d-1)=cd-2.
\end{align*}
Therefore $\dim J_1\leq cd-2$. It follows that
$$
    \dim I_t'' \leq \max\{\dim J_0,\, \dim J_1\}\leq cd - 2,
$$
which proves the desired bound.
\end{proof}

\begin{lemma}
\label{lem:ext-tan}
$\dim X_t\leq cd - 2$, under the assumption that $Y_\delta$ is not a hypersurface.
\end{lemma}
\begin{proof}
Due to Lemma~\ref{lem:L-tan-at-tan}, it remains to prove that
$$
    \dim X_t^\circ\leq cd - 2.
$$
Each $L\in X_t^\circ$ is tangent at both $P_0$ and some point $P\in Y_\delta^{\mathrm{sm}}\setminus \{P_0\}$. Hence there exist lines
$$
    \ell_0\subseteq L\cap T_{P_0}Y_\delta
    \qquad\text{and}\qquad
    \ell\subseteq L\cap T_PY_\delta
$$
passing through $P_0$ and $P$, respectively. This shows that the projection map
$$
    I_t\longrightarrow X_t^\circ
    : (L, \ell_0, (P, \ell))\longmapsto L
$$
is surjective. It follows that
$$
    \dim X_t^\circ\leq \dim I_t
    \leq \max\{\dim I_t',\, \dim I_t''\},
$$
which is at most $cd - 2$ by Lemmas~\ref{lem:ext-tan_disjoint} and \ref{lem:ext-tan_a-point}.
\end{proof}

\begin{thm}
\label{thm:det-var_quasireflex}
Every determinantal variety $Y_\delta\subseteq\bP^N$ is quasireflexive. That is, there exists a linear subspace $L\subseteq\bP^N$ of dimension $N - \dim Y_\delta$ that intersects $Y_\delta$ in $\deg(Y_\delta) - 1$ points at which $Y_\delta$ is smooth.
\end{thm}

\begin{proof}
The hypersurface case is proved in Proposition~\ref{prop:det-hyper}. Assume that $Y_\delta$ is not a hypersurface. By Lemma~\ref{lem:multiplicity-2-open}, there exists a Zariski open subset $U\subseteq X_0$ such that $Y_{\delta}\cap L$ has multiplicity exactly two at $P_0$ for each $L\in U$.

The constructible locus $X_t$ and its closure have the same dimension. By Lemmas~\ref{lem:L-tan-to-P0}, \ref{lem:meet-sing}
and \ref{lem:ext-tan}, $X_s\cup\overline{X_t}$ is a proper closed subset of the irreducible variety $X_0$. Choose any
$$
    L\in U\setminus(X_s\cup\overline{X_t}).
$$
Then the intersection $L\cap Y_\delta$ is disjoint from the singular locus $Y_{\delta-1}$ and is transverse away from the double point $P_0$. Since all intersection points are isolated, $L\cap Y_\delta$ is finite. By B\'ezout's theorem, the total intersection multiplicity is $\deg(Y_\delta)$. The point $P_0$ contributes two, whereas every other intersection point contributes one. Hence $L\cap Y_\delta$ is supported on exactly $\deg(Y_\delta)-1$ distinct smooth points, as desired.
\end{proof}

\appendix
\section{Reflexivity of determinantal varieties}
\label{sec:reflexive-det}

For irreducible projective curves of degree at least two, reflexivity and quasireflexivity are equivalent in characteristic different from $2$ \cite{Ent21}*{Proposition~2.1}. In higher dimensions, following the second paragraph of the proof of \cite{BH86}*{Theorem at p.~906}, one deduces that:

\begin{prop}
\label{prop:ref-versus-quasiref}
Let $K$ be an algebraically closed field of characteristic
different from $2$, and let $Y\subseteq\bP^N_K$ be an irreducible
projective variety of degree at least two. If $Y$ is reflexive,
then $Y$ is quasireflexive.
\end{prop}

In what follows, we prove that determinantal varieties defined over a field of arbitrary characteristic are reflexive. By Proposition~\ref{prop:ref-versus-quasiref}, this already implies their quasireflexivity in characteristic different from $2$. However, it does not address the case of characteristic $2$, which is of particular importance in coding theory.

\bigskip

We begin by recalling the definition of reflexivity. Let $Y\subseteq\bP^N$ be a projective variety defined over an algebraically closed field $K$ of arbitrary characteristic. Its \emph{conormal variety} $C(Y)\subseteq\bP^N\times\left(\bP^N\right)^*$ is the Zariski closure of the set of pairs $(P, H)\in\bP^N\times\left(\bP^N\right)^*$ where
\begin{itemize}
\item $P$ belongs to the smooth locus $Y^{\rm sm}\subseteq Y$, and
\item $H$ is a hyperplane in $\bP^N$ containing the tangent space $T_PY\subseteq\bP^N$.
\end{itemize}
Consider the two projection maps
$$\xymatrix@R=12pt@C=12pt{
    & C(Y)\ar[dl]_-{\pi_1}\ar[dr]^-{\pi_2} & \\
    Y && \left(\bP^N\right)^*.
}$$
The \emph{dual variety} of $Y$ is the image
$Y^*\colonequals\pi_2(C(Y))\subseteq(\bP^N)^*$
of the second projection, and $Y$ is \emph{reflexive} if $C(Y)\cong C(Y^{\ast})$ under the natural isomorphisms
\begin{equation}
\label{eqn:reflex}
    \bP^N \times\left(\bP^N\right)^*
    \cong 
    \left(\bP^N\right)^*\times\bP^N
    \cong
    \left(\bP^N\right)^*\times\left(\bP^N\right)^{**}.
\end{equation}

Let $Y_\delta\subseteq\bP^N = \bP^{mn-1}$ be the determinantal
variety over $K$ consisting of $m\times n$ matrices of rank
$<\delta$, where $2\leq\delta\leq n\leq m$. Recall that the tangent space at a smooth point $P\in Y_\delta^{\rm sm}$ is given by
$$
    T_PY_\delta = \left\{
        Q\in\bP^N
    \;\middle|\;
       Q(\ker P)
       \subseteq
       \operatorname{im}P
    \right\}.
$$
On the other hand, there is a nondegenerate bilinear form 
$$
    \langle -,-\rangle\colon
    K^{mn}\times K^{mn}
    \longrightarrow K,
    \quad (A, B)\longmapsto {\rm tr}(AB^T).
$$
In coordinates, if we write $A = (a_{ij})$ and $B = (b_{ij})$, then $\langle A, B\rangle = \sum_{i,j} a_{ij}b_{ij}$. We identify $(\bP^N)^*$ with $\bP^N$ via the induced isomorphism, which is explicitly given by
$$\xymatrix{
    \bP^N\ar[r]^-\sim & (\bP^N)^*
    : P\ar@{|->}[r] & H_P = \left\{
        \sum_{i,j} p_{ij}x_{ij} = 0
    \right\}
}$$
where $(p_{ij})$ is any matrix representing $P$.

\begin{lemma}
\label{lem:local_calculation}
Let $I_{\delta-1}$ denote the $(\delta-1)\times(\delta-1)$ identity matrix. The fiber over
$$
    P_0 = \begin{pmatrix}
        I_{\delta-1} & 0 \\
        0 & 0
    \end{pmatrix}\in Y_\delta
$$
under the projection $\pi_1\colon C(Y_\delta)\longrightarrow Y_\delta$ has the form
$$
    \pi_1^{-1}(P_0) = \left\{\left(
        P_0,
        \begin{pmatrix}
            0 & 0 \\
            0 & \ast_{m-\delta+1,\, n-\delta+1}
    \end{pmatrix}\right)\right\}
$$
where $\ast_{\alpha,\, \beta}$ ranges over nonzero
$\alpha\times\beta$ matrices, considered up to scalar. In particular,
$$
    \pi_2\left(
        \pi_1^{-1}(P_0)
    \right)\subseteq Y_{n-\delta+2}.
$$
\end{lemma}
\begin{proof}
Every $Q\in T_{P_0}Y_\delta$ is represented by a matrix of the form
$$
    \begin{pmatrix}
        \ast_{\delta-1,\,\delta-1}
        & \ast_{\delta-1,\,n-\delta+1} \\
        \ast_{m-\delta+1,\,\delta-1}
        & 0
    \end{pmatrix}.
$$
Hence, $H_P$ contains $T_{P_0}Y_\delta$ if and only if $P$ is represented by
$$
    \widetilde{P} = \begin{pmatrix}
        0 & 0 \\
        0 & \ast_{m-\delta+1,\,n-\delta+1}
    \end{pmatrix}.
$$
Since $(P_0, P)\in\bP^N\times(\bP^N)^*$ belongs to $\pi_1^{-1}(P_0)$ if and only if $T_{P_0}Y_\delta\subseteq H_P$, this proves the first statement. The second statement follows immediately from the expression of $\widetilde{P}$.
\end{proof}

\begin{lemma}
\label{lem:inv-tr}
For any $P\in Y_\delta^{\rm sm}$, $Q\in\bP^N$,
$A\in\GL_m(K)$ and $B\in\GL_n(K)$,
$$
    T_PY_\delta\subseteq H_Q
    \;\Longleftrightarrow\;
    T_{APB}Y_\delta
    \subseteq
    H_{A^\dagger Q B^\dagger}
$$
where $A^\dagger = (A^{-1})^T$ and $B^\dagger = (B^{-1})^T$.
\end{lemma}
\begin{proof}
Note that
\begin{equation}
\label{eqn:trace}
    {\rm tr}\left(
        P\,Q^T
    \right)
    = {\rm tr}\left(
        (APB)\,
        (A^\dagger QB^\dagger)^T
    \right).
\end{equation}
Using the identities
\begin{itemize}
\item $\ker(APB)
    = \ker(PB)=B^{-1} \ker P$,
\item $\operatorname{im}(APB)
    = \operatorname{im}(AP)
    = A\,\operatorname{im}P$,
\end{itemize}
we deduce for any $M\in\Hom(K^n, K^m)$ that
\begin{align*}
    (AMB)\,\ker(APB)
        \subseteq\operatorname{im}(APB)
    &\;\Longleftrightarrow\;
        (AMB)\,(B^{-1}\ker P)
        \subseteq A\,\operatorname{im}P \\
    &\;\Longleftrightarrow\;
        M\,\ker P
        \subseteq\operatorname{im}P.
\end{align*}
It follows that
\begin{equation}
\label{eqn:ker-im}
    A\,(T_PY_\delta)\,B = T_{APB}Y_\delta.
\end{equation}
Consequently,
\begin{align*}
    T_PY_\delta\subseteq H_Q
    & \;\Longleftrightarrow\;
        \left<
            T_PY_\delta,\, Q
        \right> = 0 \\
        \tag{by \eqref{eqn:trace}}
    &\;\Longleftrightarrow\;
        \left<
            A\,(T_PY_\delta)\, B,\,
            A^\dagger Q B^\dagger
        \right> = 0 \\
        \tag{by \eqref{eqn:ker-im}}
    & \;\Longleftrightarrow\;
        \left<
            T_{APB}Y_\delta,\,
            A^\dagger QB^\dagger
        \right> = 0
    \;\Longleftrightarrow\;
        T_{APB}Y_\delta
        \subseteq H_{A^\dagger QB^\dagger}.
\end{align*}
This completes the proof.
\end{proof}

\begin{lemma}
\label{lem:PQT}
Let $P\in Y_\delta^{\rm sm}$ and $Q\in\bP^{N}$. Then
$$
    T_PY_\delta\subseteq H_Q
    \quad\Longleftrightarrow\quad
    PQ^T=0
    \quad\text{and}\quad
    P^TQ=0.
$$
\end{lemma}
\begin{proof}
Choose $A\in\GL_m(K)$ and $B\in\GL_n(K)$ such that
$
    APB = P_0,
$
and set
$$
    Q'\colonequals A^\dagger QB^\dagger,
    \qquad
    A^\dagger=(A^{-1})^T,
    \qquad
    B^\dagger=(B^{-1})^T.
$$
By Lemma~\ref{lem:inv-tr},
$$
    T_PY_\delta\subseteq H_Q
    \quad\Longleftrightarrow\quad
    T_{P_0}Y_\delta\subseteq H_{Q'}.
$$
By Lemma~\ref{lem:local_calculation}, the latter condition holds if and only if $Q'$ has only a lower-right block. This is equivalent to
$$
    P_0(Q')^T = 0
    \qquad\text{and}\qquad
    P_0^TQ' = 0.
$$
On the other hand,
$$
    P_0(Q')^T
    = (APB)(A^\dagger QB^\dagger)^T
    = A(PQ^T)A^{-1}
$$
and
$$
    P_0^TQ'
    =
    (APB)^T(A^\dagger QB^\dagger)
    =
    B^T(P^TQ)(B^T)^{-1}.
$$
Thus, the last condition is equivalent to $PQ^T = 0$ and $P^TQ = 0$.
\end{proof}

\begin{prop}
\label{prop:det-var_reflex}
For any $(P, Q)\in Y_\delta^{\rm sm}\times Y_{n-\delta+2}^{\rm sm}$,
$$
    T_PY_\delta\subseteq H_Q
    \quad\Longleftrightarrow\quad
    T_QY_{n-\delta+2}\subseteq H_P.
$$
In particular, $Y_\delta$ has dual variety $Y_{n-\delta+2}$ and is reflexive.
\end{prop}
\begin{proof}
Set $\delta' = n-\delta+2$. By Lemma~\ref{lem:PQT},
$$
    T_PY_\delta\subseteq H_Q
    \quad\Longleftrightarrow\quad
    PQ^T=0
    \quad\text{and}\quad
    P^TQ=0.
$$
Applying the same lemma with $P$ and $Q$ reversed gives
$$
    T_QY_{\delta'}\subseteq H_P
    \quad\Longleftrightarrow\quad
    QP^T=0
    \quad\text{and}\quad
    Q^TP=0.
$$
Therefore, the condition $T_PY_\delta\subseteq H_Q$ is equivalent to $T_QY_{\delta'}\subseteq H_P$.

Interchanging the factors gives an isomorphism
$$
    \left\{
        (P, Q)\in Y_{\delta}^{\rm sm}
            \times Y_{\delta'}^{\rm sm}
    \;\middle|\;
        T_PY_\delta\subseteq H_Q
    \right\}
    \cong
    \left\{
        (Q, P)\in Y_{\delta'}^{\rm sm}
            \times Y_{\delta}^{\rm sm}
    \;\middle|\;
        T_QY_{\delta'}\subseteq H_P
    \right\}.
$$
By Lemmas~\ref{lem:local_calculation} and \ref{lem:inv-tr}, each fiber over $Y_\delta^{\rm sm}$ under the map $C(Y_\delta)\to Y_\delta$ is the projective space of $(m-\delta+1)\times(n-\delta+1)$ matrices. Since $m\geq n$, matrices of full column rank form a nonempty open dense subset. Within each fiber, these correspond precisely to the points with $Q\in Y_{\delta'}^{\rm sm}$. Thus, the incidence locus on the left is dense in $C(Y_\delta)$. As $n-\delta'+2 = \delta$, the same argument shows that the locus on the right is dense in $C(Y_{\delta'})$.

Taking closures shows that interchanging the factors identifies $C(Y_\delta)$ with $C(Y_{\delta'})$. Projecting to the factors gives $Y_\delta^*=Y_{\delta'}$ and $Y_{\delta'}^*=Y_\delta$. In particular, $C(Y_\delta)\cong C(Y_\delta^*)$, so $Y_\delta$ is reflexive.
\end{proof}

\bibliographystyle{alpha}
\bibliography{main}

\begin{bibdiv}
\begin{biblist}

\bib{AGL19}{article}{
      author={Antrobus, Jared},
      author={Gluesing-Luerssen, Heide},
       title={Maximal {F}errers diagram codes: constructions and genericity
  considerations},
        date={2019},
        ISSN={0018-9448,1557-9654},
     journal={IEEE Trans. Inform. Theory},
      volume={65},
      number={10},
       pages={6204\ndash 6223},
         url={https://doi.org/10.1109/TIT.2019.2926256},
      review={\MR{4009199}},
}

\bib{BH86}{article}{
      author={Ballico, E.},
      author={Hefez, A.},
       title={On the {G}alois group associated to a generically \'{e}tale
  morphism},
        date={1986},
        ISSN={0092-7872},
     journal={Comm. Algebra},
      volume={14},
      number={5},
       pages={899\ndash 909},
         url={https://doi.org/10.1080/00927878608823344},
      review={\MR{834472}},
}

\bib{BHLPRW22}{article}{
      author={Bartz, Hannes},
      author={Holzbaur, Lukas},
      author={Liu, Hedongliang},
      author={Puchinger, Sven},
      author={Renner, Julian},
      author={Wachter-Zeh, Antonia},
       title={Rank-metric codes and their applications},
    language={English},
        date={2022},
        ISSN={1567-2190},
     journal={Found. Trends Commun. Inf. Theory},
      volume={19},
      number={3},
       pages={390\ndash 546},
}

\bib{BR20}{article}{
      author={Byrne, Eimear},
      author={Ravagnani, Alberto},
       title={Partition-balanced families of codes and asymptotic enumeration
  in coding theory},
        date={2020},
        ISSN={0097-3165,1096-0899},
     journal={J. Combin. Theory Ser. A},
      volume={171},
       pages={105169, 39},
         url={https://doi.org/10.1016/j.jcta.2019.105169},
      review={\MR{4035968}},
}

\bib{CGH23}{article}{
      author={Cluckers, Raf},
      author={Glazer, Itay},
      author={Hendel, Yotam~I.},
       title={A number theoretic characterization of {$E$}-smooth and ({FRS})
  morphisms: estimates on the number of {$\mathbb{Z}/p^k\mathbb{Z}$}-points},
        date={2023},
        ISSN={1937-0652,1944-7833},
     journal={Algebra Number Theory},
      volume={17},
      number={12},
       pages={2229\ndash 2260},
         url={https://doi.org/10.2140/ant.2023.17.2229},
      review={\MR{4650394}},
}

\bib{CMPZ17}{article}{
      author={Csajb{\'o}k, Bence},
      author={Marino, Giuseppe},
      author={Polverino, Olga},
      author={Zullo, Ferdinando},
       title={Maximum scattered linear sets and {MRD}-codes},
        date={2017},
        ISSN={0925-9899,1572-9192},
     journal={J. Algebraic Combin.},
      volume={46},
      number={3-4},
       pages={517\ndash 531},
         url={https://doi.org/10.1007/s10801-017-0762-6},
}

\bib{Ent21}{article}{
      author={Entin, Alexei},
       title={Monodromy of hyperplane sections of curves and decomposition
  statistics over finite fields},
        date={2021},
        ISSN={1073-7928,1687-0247},
     journal={Int. Math. Res. Not. IMRN},
      number={14},
       pages={10409\ndash 10441},
         url={https://doi.org/10.1093/imrn/rnz120},
      review={\MR{4285725}},
}

\bib{Ful98}{book}{
      author={Fulton, William},
       title={Intersection theory},
     edition={Second edition},
   publisher={Springer-Verlag, Berlin},
        date={1998},
        ISBN={3-540-62046-X; 0-387-98549-2},
         url={https://doi.org/10.1007/978-1-4612-1700-8},
}

\bib{GHRW24}{article}{
      author={Gruica, Anina},
      author={Horlemann, Anna-Lena},
      author={Ravagnani, Alberto},
      author={Willenborg, Nadja},
       title={Densities of codes of various linearity degrees in
  translation-invariant metric spaces},
        date={2024},
        ISSN={0925-1022,1573-7586},
     journal={Des. Codes Cryptogr.},
      volume={92},
      number={3},
       pages={609\ndash 637},
         url={https://doi.org/10.1007/s10623-023-01236-2},
      review={\MR{4726088}},
}

\bib{GJLR20}{article}{
      author={Gorla, Elisa},
      author={Jurrius, Relinde},
      author={L{\'o}pez, Hiram~H.},
      author={Ravagnani, Alberto},
       title={Rank-metric codes and $q$-polymatroids},
        date={2020},
        ISSN={0925-9899,1572-9192},
     journal={J. Algebraic Combin.},
      volume={52},
      number={1},
       pages={1\ndash 19},
         url={https://doi.org/10.1007/s10801-019-00889-4},
}

\bib{GKR23}{article}{
      author={Gruica, Anina},
      author={K{\i}l{\i}{\c{c}}, Altan~B.},
      author={Ravagnani, Alberto},
       title={Rank-metric codes and their parameters},
        date={2023},
     journal={\href{https://arxiv.org/abs/2312.06282}{arXiv:2312.06282}},
}

\bib{GL20}{article}{
      author={Gluesing-Luerssen, Heide},
       title={On the sparseness of certain linear {MRD} codes},
        date={2020},
        ISSN={0024-3795},
     journal={Linear Algebra Appl.},
      volume={596},
       pages={145\ndash 168},
         url={https://doi.org/10.1016/j.laa.2020.03.006},
      review={\MR{4076988}},
}

\bib{GR22}{article}{
      author={Gruica, Anina},
      author={Ravagnani, Alberto},
       title={Common complements of linear subspaces and the sparseness of
  {MRD} codes},
        date={2022},
     journal={SIAM J. Appl. Algebra Geom.},
      volume={6},
      number={2},
       pages={79\ndash 110},
         url={https://doi.org/10.1137/21M1428947},
      review={\MR{4405182}},
}

\bib{GRSZ23}{article}{
      author={Gruica, Anina},
      author={Ravagnani, Alberto},
      author={Sheekey, John},
      author={Zullo, Ferdinando},
       title={Rank-metric codes, semifields, and the average critical problem},
        date={2023},
        ISSN={0895-4801,1095-7146},
     journal={SIAM J. Discrete Math.},
      volume={37},
      number={2},
       pages={1079\ndash 1117},
         url={https://doi.org/10.1137/22M1486893},
      review={\MR{4602503}},
}

\bib{Har92}{book}{
      author={Harris, Joe},
       title={Algebraic geometry},
      series={Graduate Texts in Mathematics},
   publisher={Springer-Verlag, New York},
        date={1992},
      volume={133},
        ISBN={0-387-97716-3},
         url={https://doi.org/10.1007/978-1-4757-2189-8},
        note={A first course},
      review={\MR{1182558}},
}

\bib{HHLT22}{article}{
      author={Hao, Jing},
      author={Huang, Han},
      author={Livshyts, Galyna~V.},
      author={Tikhomirov, Konstantin},
       title={Distribution of the minimum distance of random linear codes},
        date={2022},
        ISSN={0018-9448,1557-9654},
     journal={IEEE Trans. Inform. Theory},
      volume={68},
      number={10},
       pages={6388\ndash 6401},
         url={https://doi.org/10.1109/tit.2022.3170341},
      review={\MR{4502542}},
}

\bib{Kai14}{article}{
      author={Kaipa, Krishna~V.},
       title={An asymptotic formula in {$q$} for the number of {$[n,k]$}
  {$q$}-ary {MDS} codes},
        date={2014},
        ISSN={0018-9448,1557-9654},
     journal={IEEE Trans. Inform. Theory},
      volume={60},
      number={11},
       pages={7047\ndash 7057},
         url={https://doi.org/10.1109/TIT.2014.2349903},
      review={\MR{3273040}},
}

\bib{KS22}{article}{
      author={Kmentt, Philip},
      author={Shute, Alec},
       title={The {B}ertini irreducibility theorem for higher codimensional
  slices},
        date={2022},
        ISSN={1071-5797,1090-2465},
     journal={Finite Fields Appl.},
      volume={83},
       pages={Paper No. 102085, 7},
         url={https://doi.org/10.1016/j.ffa.2022.102085},
      review={\MR{4456450}},
}

\bib{Loi14}{article}{
      author={Loidreau, P.},
       title={Asymptotic behaviour of codes in rank metric over finite fields},
        date={2014},
        ISSN={0925-1022,1573-7586},
     journal={Des. Codes Cryptogr.},
      volume={71},
      number={1},
       pages={105\ndash 118},
         url={https://doi.org/10.1007/s10623-012-9716-0},
      review={\MR{3167050}},
}

\bib{LW54}{article}{
      author={Lang, Serge},
      author={Weil, Andr\'e},
       title={Number of points of varieties in finite fields},
        date={1954},
        ISSN={0002-9327,1080-6377},
     journal={Amer. J. Math.},
      volume={76},
       pages={819\ndash 827},
         url={https://doi.org/10.2307/2372655},
      review={\MR{65218}},
}

\bib{NHTRR18}{article}{
      author={Neri, Alessandro},
      author={Horlemann-Trautmann, Anna-Lena},
      author={Randrianarisoa, Tovohery},
      author={Rosenthal, Joachim},
       title={On the genericity of maximum rank distance and {G}abidulin
  codes},
        date={2018},
        ISSN={0925-1022,1573-7586},
     journal={Des. Codes Cryptogr.},
      volume={86},
      number={2},
       pages={341\ndash 363},
         url={https://doi.org/10.1007/s10623-017-0354-4},
      review={\MR{3756146}},
}

\bib{PS22}{article}{
      author={Poonen, Bjorn},
      author={Slavov, Kaloyan},
       title={The exceptional locus in the {B}ertini irreducibility theorem for
  a morphism},
        date={2022},
        ISSN={1073-7928},
     journal={Int. Math. Res. Not. IMRN},
      number={6},
       pages={4503\ndash 4513},
         url={https://doi.org/10.1093/imrn/rnaa182},
      review={\MR{4391895}},
}

\bib{Sha13}{book}{
      author={Shafarevich, Igor~R.},
       title={Basic algebraic geometry. 1},
     edition={Third edition},
   publisher={Springer, Heidelberg},
        date={2013},
        ISBN={978-3-642-37955-0; 978-3-642-37956-7},
        note={Varieties in projective space},
      review={\MR{3100243}},
}

\bib{SilvaK11}{article}{
      author={Silva, Danilo},
      author={Kschischang, Frank~R.},
       title={Universal secure network coding via rank-metric codes},
        date={2011},
        ISSN={0018-9448,1557-9654},
     journal={IEEE Trans. Inform. Theory},
      volume={57},
      number={2},
       pages={1124\ndash 1135},
         url={https://doi.org/10.1109/TIT.2010.2090212},
}

\bib{SilvaKS08}{article}{
      author={Silva, Danilo},
      author={Kschischang, Frank~R.},
      author={K{\"o}tter, Ralf},
       title={A rank-metric approach to error control in random network
  coding},
        date={2008},
        ISSN={0018-9448,1557-9654},
     journal={IEEE Trans. Inform. Theory},
      volume={54},
      number={9},
       pages={3951\ndash 3967},
         url={https://doi.org/10.1109/TIT.2008.928291},
}

\end{biblist}
\end{bibdiv}

\ContactInfo
\end{document}